\documentclass[aop]{imsart}

\usepackage{amsthm,amsmath}

\RequirePackage[numbers]{natbib}

\usepackage{amsfonts}
\usepackage{amssymb}
\usepackage{enumerate}
\usepackage{hyperref}
\usepackage{fancyhdr}
\usepackage{mathrsfs} 
\usepackage{tikz}
\usepackage{stmaryrd}

\startlocaldefs

\newtheorem{theorem}{Theorem}
\newtheorem{lemma}{Lemma}
\newtheorem{proposition}[lemma]{Proposition}
\newtheorem{corollary}[lemma]{Corollary}
\newtheorem{definition}[lemma]{Definition}
\newtheorem{remark}[lemma]{Remark}

\newtheorem{assumption}{Assumption}

\newcommand{\heta}{\hat{\eta}}
\newcommand{\hzeta}{\hat{\zeta}}

\newcommand{\Be}{\mathrm{Be}}

\newcommand{\E}{\mathbb{E}}
\renewcommand{\P}{\mathbb{P}}
\newcommand{\tun}{\mathtt{1}}

\newcommand{\bbC}{\mathbb{C}}
\newcommand{\bbD}{\mathbb{D}}
\newcommand{\bbE}{\mathbb{E}}

\newcommand{\bbN}{\mathbb{N}}

\newcommand{\bbP}{\mathbb{P}}
\newcommand{\bbQ}{\mathbb{Q}}
\newcommand{\bbR}{\mathbb{R}}

\newcommand{\bbZ}{\mathbb{Z}}

\newcommand{\cC}{\mathcal{C}}

\newcommand{\cF}{\mathcal{F}}
\newcommand{\cG}{\mathcal{G}}

\newcommand{\cL}{\mathcal{L}}
\newcommand{\cM}{\mathcal{M}}

\newcommand{\cS}{\mathcal{S}}

\newcommand{\ccM}{\mathscr{M}}

\newcommand{\gep}{\varepsilon}       

\newcommand{\gD}{\Delta}

\newcommand{\gO}{\Omega}
\newcommand{\gl}{\lambda}

\newcommand{\ind}{\mathbf{1}}
\DeclareMathOperator{\gap}{\mathrm{gap}}
\DeclareMathOperator{\Airy}{\mathrm{Airy}}

\DeclareMathSymbol{\leqslant}{\mathalpha}{AMSa}{"36} 
\DeclareMathSymbol{\geqslant}{\mathalpha}{AMSa}{"3E} 
\DeclareMathSymbol{\eset}{\mathalpha}{AMSb}{"3F}     
\renewcommand{\leq}{\;\leqslant\;}                   
\renewcommand{\geq}{\;\geqslant\;}                   
\newcommand{\dd}{\,\text{\rm d}}             
\newcommand{\Id}{\mathrm{id}}

\newcommand{\mintwo}[2]{\min_{\substack{#1 \\ #2}}} 
\newcommand{\suptwo}[2]{\sup_{\substack{#1 \\ #2}}} 
\newcommand{\inftwo}[2]{\inf_{\substack{#1 \\ #2}}} 
\newcommand{\limtwo}[2]{\lim_{\substack{#1 \\ #2}}}     
\newcommand{\lint}{\llbracket}
\newcommand{\llb}{\llbracket}
\newcommand{\rrb}{\rrbracket}
\newcommand{\rint}{\rrbracket}

\newcommand{\Tm}{T_{\rm mix}}

\endlocaldefs

\begin{document}

\begin{frontmatter}

\title{Cutoff phenomenon for the asymmetric\\
simple exclusion process
and \\ the biased card shuffling}
\runtitle{Cutoff for the ASEP}

\begin{aug}
\author{\fnms{Cyril} \snm{Labb\'e}\thanksref{m1}\ead[label=e1]{labbe@ceremade.dauphine.fr}}
\and
\author{\fnms{Hubert} \snm{Lacoin}\thanksref{m2}\ead[label=e2]{lacoin@impa.br}}
\address{PSL Research University, Ceremade, 75775 Paris Cedex 16, France. \printead{e1}}
\address{IMPA, Estrada Dona Castorina 110, Rio de Janeiro, Brasil. \printead{e2}}
\affiliation{Universit\'e Paris-Dauphine\thanksmark{m1} and IMPA\thanksmark{m2}}
\end{aug}
\runauthor{C.~Labb\'e and H.~Lacoin}

\begin{abstract}
We consider the biased card shuffling and the Asymmetric Simple Exclusion Process (ASEP) on the segment. 
We obtain the asymptotic of their mixing times: our result show that these two continuous-time Markov chains display cutoff.
Our analysis combines several ingredients including: a study of the hydrodynamic profile for ASEP, the use of monotonic eigenfunctions, stochastic comparisons and concentration inequalities.
\end{abstract}



 \maketitle
\end{frontmatter}


\setcounter{tocdepth}{1}

\tableofcontents

\section{Introduction}

The relaxation to equilibrium for particle systems is a subject that has given rise to a rich literature.
The phenomenon has been studied from different viewpoints: importance was first given to 
the problem of the evolution of the particle density on a macroscopic space scale and an adequate time scale, which is usually referred to 
as hydrodynamic limits (see \cite{KipLan} for a detailed account on the subject as well as references), 
but the modern theory of Markov chains highlighted another aspect of the problem, 
which is how the equilibrium state is approached in terms of distance between probability measures, or the mixing time problem \cite{LevPerWil}.

\smallskip

While hydrodynamic limits are now fairly well understood for the simplest particle systems (exclusion process, zero range etc...), 
there are still some fundamental questions on the mixing time that remain unsolved. In particular, it is believed that for many particle systems
convergence to equilibrium occurs abruptly, a phenomenon known as cutoff
(see \cite[Chapter 18]{LevPerWil} for an introduction to cutoff and examples of Markov chains with cutoff). 
Until now, this has been rigorously proved to hold only for some specific cases among which
the simple (symmetric) exclusion process on the complete graph \cite{DiSh87} or in one dimensional graphs (segment and circle) \cite{Lac16, Lac162}.

In many other cases a weaker version of the statement, called pre-cutoff, has been proved (see below for a precise definition). This includes for instance the process which is the focus of this paper: the Asymmetric Simple Exclusion Process (ASEP) on the segment \cite{Benjamini}
(and also more recently \cite{GPR09}).

The ASEP can be defined as follows: $k$ particles on a segment of length 
$N$ jump independently with rate $p>1/2$ to the right and $q=(1-p)$ to the left.
A restriction is added: each site can be occupied by at most one particle,
so that every jump which yields a configuration that violates this restriction is canceled. 
We also study the biased card shuffling which is a walk on the symmetric group 
from which the ASEP can be obtained as a projection; this Markov chain also displays pre-cutoff.

While it has been known, since a counter example was proposed by Aldous in 2004 (see \cite[Figure 18.2]{LevPerWil}) that it is possible 
to have pre-cutoff without cutoff,
it is a folklore conjecture in the field that all ``reasonable'' Markov chains with pre-cutoff should in fact have cutoff, thus providing many open problems (see \cite[Section 23.2]{LevPerWil}).

The main achievement of this paper is to show that indeed cutoff holds for the ASEP and to identify the asymptotic behavior for the mixing time. This solves a question which had been
left open since the publication of \cite{Benjamini}. We prove that the mixing time corresponds exactly to the time at which the particle density stabilizes to the equilibrium profile: this underlines the connection between hydrodynamic limits and mixing times (which was already underlined in the 
symmetric case see
e.g. \cite{Lac162, LeeYau98}).
We also derive a similar result for the biased card shuffling.

Note that while the hydrodynamic limit for the asymmetric exclusion process on the full line has been well understood for many years 
\cite{Reza} (also \cite{Benfou, Liggettbook, Rost81} for the special ``wedge'' initial condition), 
the presence of boundary conditions makes the problem more delicate to analyze, and a substantial part of our paper is devoted to
the analysis of the scaling limits of two quantities associated with the ASEP:
\begin{itemize}
 \item the particle density (which had been analyzed in the small biased case by one of the authors \cite{LabbeKPZ}),
 \item the positions of the leftmost particle and the rightmost empty site (which, depending on the initial condition, 
 may or may not coincide with what is suggested by the limit of the particle density).
\end{itemize}

\section{Model and results}

\subsection{Biased card shuffling}

Given $N\in \bbN$ and $p\in (1/2,1]$, we set $q=1-p$  and consider the following continuous time Markov chain 
on the set of permutations of $N$ cards labeled from one to $N$.
Each pair of adjacent cards is chosen at rate one: then, with probability $p$ (corresponding to an independent Bernoulli 
random variable) we arrange the cards such that the lower card comes before the higher card and with probability $q$ we arrange them so that the 
higher card comes first.

A configuration of cards can be represented by an element $\sigma$ of the symmetric group $\cS_N$: 
for every $i\in \lint 1,N\rint$, $\sigma(i)$ (we use the notation $\lint a,b \rint=[a,b]\cap \bbZ$) denotes the label of the card at position $i$. The dynamics presented above then corresponds to the Markov process on $\cS_N$ with the following generator:
\begin{equation}\label{generator}
\begin{split}
\cL_N f(\sigma)&:= \sum_{i=1}^{N-1} \left( p\ind_{\{\sigma(i+1) < \sigma(i)\}} + q\ind_{\{\sigma(i+1) > \sigma(i)\}} \right)[f(\sigma\circ\tau_i))-f(\sigma)]\\
 &=\sum_{i=1}^{N-1} p[f(\sigma^{i,+})-f(\sigma)]+q[f(\sigma^{i,-})-f(\sigma)].
\end{split}\end{equation}
In the expression above, $\tau_i$ denotes the transposition $(i,i+1)$ and $\sigma^{i,+}, \sigma^{i,-}$ denote the elements of $\cS_N$ which satisfy the following property
\begin{equation*}
 \begin{cases}
  \sigma^{i,\pm}(j)&= \sigma(j), \qquad \forall j\in \lint 1,N \rint \setminus\{i,i+1\}\;,\\
  \sigma^{i,+}(i)&<   \sigma^{i,+}(i+1)\;,\\
  \sigma^{i,-}(i)&>   \sigma^{i,-}(i+1)\;.
 \end{cases}
\end{equation*}

\noindent Note that either $\sigma^{i,+}$  or $\sigma^{i,-}$ is equal to $\sigma$ so that the choice of a pair of cards does not always imply a modification of the permutation.

As $p>q$, this way of shuffling cards favors permutations which are more ``ordered''.
More precisely, if we let $D(\sigma)$ denote the minimal number of adjacent transpositions needed to obtain $\sigma$  starting from the identity permutation 
(the graph distance between $\sigma$ and the identity in the Cayley graph of $\cS_N$ with generator $(\tau_i)_{i=1}^{N-1}$), then one can check that
the equilibrium measure is given by 
\begin{equation*}
 \pi_N(\sigma):=\frac{\gl^{-D(\sigma)}}{\sum_{\sigma' \in \cS_N}\gl^{-D(\sigma')}}\;,
\end{equation*}
where $\gl=p/q$. The detailed balance condition is easy to check with the relation
$$D(\sigma):=\sum_{i<j} \ind_{\{\sigma(i)>\sigma(j)\}}\;.$$ 
In the particular case $p=1$, the parameter $\gl$ equals $+\infty$ and the equilibrium measure $\pi_N$ is a Dirac measure at the identity permutation.

\smallskip

We denote the process starting from initial condition $\xi\in \cS_N$ by $\sigma^{\xi}_t$ and 
let $Q^{\xi}_t$ denote the distribution of $\sigma^{\xi}_t$ at time $t$.

\smallskip

Recall that the total variation distance between two probability measures $\alpha$ and $\beta$ on some discrete space $\Omega$ is defined by 
\begin{equation*}\begin{split}
 \| \alpha-\beta \|_{TV}&:= \frac{1}{2}\sum_{x\in \gO} |\alpha(x)-\beta(x)|=  \max_{A\subset \gO}[ \alpha(A)-\beta(A)]\\
 &=\inftwo{X_1\sim \alpha}{X_2\sim \beta}P(X_1\ne X_2)\;,
\end{split}\end{equation*}
where the infimum is taken over all couplings that give distribution $\alpha$ to $X_1$ and $\beta$ to $X_2$.
The fact that the three definitions are equivalent is a standard property see e.g.~\cite[Section 4.1]{LevPerWil}.
Using standard terminology, we define the (worst-case) total-variation distance to equilibrium by
\begin{equation*}
 d^{N}(t):= \max_{\xi \in \cS_N}\| Q^{\xi}_t-\pi_N \|_{TV},
\end{equation*}
and the corresponding mixing time by 
\begin{equation*}
\Tm^N(\gep):= \inf\{ t\geq 0 \ : \ d^{N}(t)<\gep\}.
\end{equation*}
A notion very much related to mixing time is that of \textsl{cutoff}, which designates a form of abrupt convergence to equilibrium for Markov chain.
For an arbitrary sequence of Markov chains, cutoff is said to hold if for all $\gep>0$,
\begin{equation*}
 \lim_{N\to \infty} \frac{\Tm^N(\gep)-\Tm^N(1-\gep)}{\Tm^N(1/4)}=0.
\end{equation*}
If  $\sup_{\gep\in(0,1)}\limsup_{N\to \infty} (\Tm^N(\gep)-\Tm^N(1-\gep))/\Tm^N(1/4)<\infty$ we say that pre-cutoff holds.

We define $\gap_{N}$ to be the spectral gap for this Markov chain. Recall that for a continuous time reversible, irreducible Markov chain with generator $\cL$ on a finite state space, the spectral gap is simply defined as 
the smallest positive eigenvalue of $-\cL$ \cite[Section 20.3]{LevPerWil}.
The spectral gap controls the asymptotic rate of convergence to equilibrium (see \cite[Corollary 12.6]{LevPerWil}), in the case of our Markov chains this gives
\begin{equation}\label{symptogap}
  \lim_{t\to \infty} \frac{1}{t} \log d^{N}(t)= -\gap_N.
\end{equation}
Our main result is the following

\begin{theorem}\label{th:shuffle}
 We have for every $p\in (1/2,1]$ and $\gep\in(0,1)$
 \begin{equation*}
 \lim_{N\to \infty} \frac{\Tm^N(\gep)}{N}=\frac{2}{p-q}.
 \end{equation*}
Moreover we have for every value of $N$ and $p$ 
\begin{equation}\label{gapbias}
 \gap_N= (\sqrt{p}-\sqrt{q})^2+ 4\sqrt{pq}   \sin  \left(\frac \pi{2N} \right)^2.
\end{equation}
\end{theorem}

Let us stress that another 
proof of \eqref{gapbias} provided by Levin and Peres in \cite{LevPer16} appeared while we were in the process of writing the present paper.

\begin{remark}
 Note that $\gap_N$ coincides with the spectral gap of the biased walk with transition rates $p$ and $q$ on the segment (and this will also be the case for the ASEP). This result is reminiscent of Aldous' spectral gap conjecture, now a theorem proved by Caputo, Liggett and Richthammer \cite{CapLigRic10}, 
 which states that the spectral gap for the interchange process on an arbitrary graph equals that of the corresponding random walk.
 However let us stress that the biased card shuffling is not an interchange process, and that our results can not be deduced from the one in \cite{CapLigRic10}.
\end{remark}

\begin{remark}
Observe that $\gap_N \Tm^N(\gep) \to \infty$ as $N\to\infty$, and recall that this condition is necessary (but not sufficient) for having a pre-cutoff, see~\cite[Sec 18.3]{LevPerWil}.
\end{remark}

Note that intuitively, as the shuffle tends to order the pack, the worst initial condition should be the permutation $\sigma_{\max}$ defined by 
\begin{equation}\label{defsigmamax}
 \sigma_{\max}:i\mapsto N+1-i \;.
\end{equation}
Our proof implies indeed that this is asymptotically the case.

\subsection{Asymmetric Simple Exclusion Process}

Given $k\in\lint 1,N-1 \rint$, we obtain another Markov process if we decide to follow only the positions of the cards labeled from $N-k+1$ to $N$, 
that is if we consider the image of 
$(\sigma_t)_{t\ge 0}$ by the transformation
\begin{equation}\label{projek}
\sigma \mapsto \ind_{\lint N-k+1,N\rint} \circ \sigma.
\end{equation}
(Our choice of following the particles with higher rather than lower labels may seem unnatural, but is driven by
the fact that we want the particles to drift to the right).

It is not difficult to check that the process obtained via this transformation is indeed Markov. A more intuitive description is the following. Consider
$k$ particles on the segment $I_N:=\lint 1,N \rint$ which are initially placed on $k$ distinct sites. The particles perform independent, continuous time, random walks on $I_N$ with 
jump rate $p$ to the right and $q$ to the left (a particle at site $1$, resp.~$N$, is not allowed to jump to its left, resp.~right): however, if a particle tries to jump on an occupied site, the jump is cancelled.

\smallskip

Denoting the presence of particle by $1$s and their absence by $0$s,
the space of configurations associated to this process is given by 
$$\gO^0_{N,k}\:= \left\{ \eta \in \{0,1\}^{I_N} \ : \ \sum_{i=1}^N \eta(i)=k \right\}\;.$$ 
We denote the evolving particle system by
$(\eta^{\xi}(t,\cdot))_{t\geq 0}$ where $\eta^{\xi}(t,x)$ equals $1$ if there is a particle at site $x$ at time $t$, and $0$ otherwise 
while $\xi$ underlines the dependence on the initial condition. The law of $\eta^{\xi}(t)$ is denoted by $P^\xi_t$.
Now if we set for $\eta\in \gO^0_{N,k}$, 
\begin{equation}\label{def:volume}
A(\eta):= \left(\sum_{i=1}^N \eta(i) (N-i)\right) -\frac{k(k-1)}{2}\;,
\end{equation}
then the equilibrium measure $\pi_{N,k}$ for the dynamics is simply the image of $\pi_N$ by the transformation \eqref{projek},
namely
\begin{equation}\label{eq:equilibro}
 \pi_{N,k}(\eta):= \frac{\gl^{-A(\eta)}}{\sum_{\eta'\in\gO^0_{N,k}} \gl^{-A(\eta')}}\;.
\end{equation}
We also define the associated distance to equilibrium and mixing time to be respectively 
\begin{equation*}
\begin{split}
 d^{N,k}(t)&:= \max_{\xi \in \gO^0_{N,k}} \|P^{\xi}_t-\pi_{N,k} \|_{TV}\;,\\
 \Tm^{N,k}(\gep)&:= \inf\{ t\geq 0 \ : \ d^{N}(t)<\gep\}\;.
\end{split}
 \end{equation*}

We are going to compute the mixing time for the system in the limit where $k/N$ tends to $\alpha \in [0,1]$. 
Even though we allow the values $0$ and $1$ for $\alpha$, we always impose $k\ge 1$ and $k\le N-1$ in order to exclude
settings where the state-space becomes trivial ($\#\gO^0_{N,k}=1$). We use the notation  
\begin{equation*}
\limtwo{N\to \infty}{k/N\to \alpha} J(k,N)=l,  
\end{equation*}
to express that the limit of the real valued function $J(k,N)$ is $l$ for all sequences such that $k/N$ tends to $\alpha$,
or in other words
\begin{equation*}
 \lim_{\gep\to 0}\limsup_{N\to \infty}\suptwo{k\in \lint 1,N-1 \rint}{|k/N-\alpha|\le \gep} |J(k,N)-l|=0\;.
\end{equation*}

\begin{theorem}\label{th:asep}
 We have for every $p\in(1/2,1]$, every $\alpha\in[0,1]$ and every $\gep\in(0,1)$
 \begin{equation*}
 \limtwo{N\to \infty}{k/N\to \alpha} \frac{\Tm^{N,k}(\gep)}{N}= \frac{(\sqrt{\alpha}+\sqrt{1-\alpha})^2}{p-q}\;.
 \end{equation*}
 Moreover for every $N$, every $k\in \lint 1,N-1 \rint$ and every $p$ we have 
 \begin{equation}\label{gapasep}
  \gap_{N,k}= (\sqrt{p}-\sqrt{q})^2+ 4\sqrt{pq}    \sin\left( \frac\pi {2N} \right)^2\;.
  \end{equation}
\end{theorem}

By symmetry, an analogous result holds for the case $p\in[0,1/2)$. 
The behavior of the system for $p=1/2$ is very different
and was the object of a particular study \cite{Lac16} where it is shown that cutoff holds on the time scale $N^2\log N$
confirming a conjecture of Wilson \cite{Wil04}.

\begin{remark}\label{rem:sym}
We have mentioned that it was sufficient to consider the case $p\in(1/2,1]$. Let us also mention that for the proof of Theorem \ref{th:asep}, we only need to treat the case $\alpha\le 1/2$ (and in all the paper we apply this restriction).
These two facts are consequences of well known symmetry considerations, which we detail here for the sake of completeness.\\
For the biased card shuffling, we can notice (recall \eqref{defsigmamax}) that given $\xi\in \cS_N$, $\sigma^{\xi}_t\circ \sigma_{\max}$
is a card shuffling with bias $(1-p)$ and initial condition $\xi\circ \sigma_{\max}$: the distance to equilibrium is 
left unchanged under permutation of the state space, and therefore the mixing time is invariant upon reversing the bias.\\
As a consequence, for $\xi\in \gO_{N,k}^0$, the process $\eta^{\xi}_t\circ \sigma_{\max}$ is an ASEP with opposite bias, and by the same argument as above, we find that the mixing time is invariant under $p\mapsto 1-p$.
Moreover if $\mathbf{1}-\xi$ denotes the configuration where zeros and ones are swapped, 
then $\mathbf{1}-\xi \in \gO_{N,N-k}^0$ and $\mathbf{1}-\eta^\xi_t$ is an ASEP with opposite bias and
complementary density of particles. Hence, the mixing time is invariant under the map $k\mapsto N-k$.
\end{remark}

\subsection{Mixing time from an arbitrary initial condition}

Instead  of the worst-case total-variation distance to equilibrium, 
we can also consider the total-variation distance to equilibrium starting from a given configuration $\xi_N \in \Omega_{N,k}^0$:
$$ d^{\xi_N}(t) := \| P_t^{\xi_N} - \pi_{N,k} \|_{TV}\;,$$
and the associated notion of mixing time
$$ \Tm^{\xi_N}(\gep) := \inf\{ t\geq 0 \ : \ d^{\xi_N}(t)<\gep\}\;.$$

We will show that asymptotically, the mixing time depends on three characteristics of $\xi_N$:
The initial location of the leftmost particle,
the initial location of the rightmost empty site,
and the initial empirical density of particles (a probability distribution on the interval).
We thus introduce the following notation, for $\xi\in \gO^0_{N,k}$:
\begin{equation}\label{def1:lknrkn}
\begin{split}
\ell_{N}(\xi)&=\min\{ x\in \lint 1,N \rint \ : \  \xi(x)= 1 \},\\
r_{N}(\xi)&=\max\{ x\in \lint 1,N \rint  \ : \  \xi(x)= 0 \}.
\end{split}
\end{equation}
As when properly renormalized, $\ell_N$, $r_N$ and the particle density all belong to compact sets, from any sequences $k_N$ ( and $\xi_N,N\geq 1$), 
one can extract a subsequence along which the three quantities converge. 
Hence, without loss of generality, we will assume (when $k/N\to\alpha>0$) that there exist $\ell,r\in [0,1]$ and $\rho_0\in L^{\infty}([0,1])$ such that
\begin{equation}\begin{split}\label{scalingz}
& \lim_{N\to \infty} \frac{\ell_{N}(\xi_N)}{N}=\ell, \quad   \lim_{N\to \infty} \frac{r_{N}(\xi_N)}{N}=r, \\
& \lim_{N\to \infty}  \frac1{N}\sum_{x=1}^N \xi_N(x) \delta_{x/N}(dy)=\rho_0(y)\dd y.
\end{split}
\end{equation}
The convergence to $\rho_0$ is meant to hold in the weak topology, that is, for any continuous function $\varphi\in \cC([0,1])$
we have 
\begin{equation}\label{Eq:WeakCV}
\lim_{N\to \infty}  \frac1{N}\sum_{x=1}^N \xi_N(x) \varphi(x/N) = \int_{[0,1]} \rho_0(x) \varphi(x)\dd x.
\end{equation}
Note that we necessarily have $\int\rho_0(x)\dd x=\alpha$.
In the case where the limit of $k/N$ vanishes, we only require convergence for $\ell_N$ as the other two are trivial.

\smallskip

We prove that under these conditions and provided that $\ell<r$, 
the mixing time starting from $\xi_N$ also displays cutoff on scale $N$. More precisely,

\begin{theorem}\label{th:asep2}
Assuming that the sequence $(\xi_N,N\geq 1)$ of initial conditions satisfies \eqref{scalingz}, we have for all $\epsilon \in (0,1)$,
$$ \lim_{N\to \infty} \frac{\Tm^{\xi_N}(\gep)}{N}=\frac1{p-q} \max\left(t_{\rho_0}, 1-\alpha-\ell,r-1+\alpha\right)\;,$$
where $t_{\rho_0}$ is a function of the initial density.

In the case where $k/N\to 0$ and $\lim_{N\to \infty} \ell_{N}(\xi_N)/N=\ell$, we have 
$$\lim_{N\to \infty} \frac{\Tm^{\xi_N}(\gep)}{N}= \frac{1-\ell}{p-q}.$$
\end{theorem}

The definition of $t_{\rho_0}$ is given in Section \ref{sec:hydropro} where we introduce the scaling limit for the process of empirical densities:
it corresponds to the time needed by this scaling limit (the solution of the Burgers equation) to reach its steady state.  

\begin{remark}
 At the cost of introducing some extra notation, we could also state (and prove in the same manner) a counterpart of Theorem \ref{th:asep2} 
 for the biased card shuffling: as for the worst-initial condition case,
 the mixing time simply corresponds to the maximal mixing time of all the associated ASEP projections.
 \end{remark}

 \subsection{Some connections with the literature}

Benjamini et al.~\cite{Benjamini} showed that the mixing time of the biased card shuffling is at most of order $N$ ( or rather $N^2$ in the discrete 
time setup considered therein, see also~\cite{GPR09}). As a lower bound matching up to a constant $\left(\frac{(1-\alpha)N}{p-q}\right)$ can be obtained 
by bounding the travel time for the leftmost particle to come to equilibrium (this argument is given in details in \cite[Section 5]{LevPer16} in the case of small bias), this established pre-cutoff.
 
 As shown in \cite{LevPer16} and in the present work, the mixing time for the ASEP and the biased card shuffling is proportional to 
 $(p-q)^{-1} N$ whenever $p>1/2$ is fixed. 
 On the other hand when $p=1/2$, we know since Wilson \cite{Wil04}, that the mixing time is order $N^2\log N$.
 A natural question which was answered by Levin and Peres in \cite{LevPer16} is: 
 if the asymmetry scales down to zero, how is the expression for the mixing time modified ?
The answer, which is given in \cite{LevPer16}, is that the expression for the mixing time depends 
on the amount of asymmetry distinguishing between three cases: $(p-q)\ge N^{-1}\log N$, $(p-q)\in [N^{-1}, N^{-1}\log N]$ and $(p-q)\le N^{-1}$.
For each of them, an expression for the mixing time is given and pre-cutoff is proved.

As cutoff holds in both the symmetric and the fully asymmetric cases, it is tempting to conjecture that it should hold
for all regimes in between. This problem is the object of a work in preparation \cite{LabLac}.

%

 Let us conclude by mentioning the references \cite{BPMRS,HadWin16} where other types of biased adjacent transpositions are considered: therein, the asymmetry is not fixed but depends on the values of 
 $\sigma(i)$ and $\sigma(i+1)$. For some families of processes of this type, rapid mixing (that is, mixing in polynomial time) is shown.
 It seems that the question of cutoff is largely open for this type chain.

\subsection{Organization of the paper}

In Section \ref{sec:techos}, we introduce alternative descriptions of the ASEP and biased card shuffling in terms of height functions, and we present some monotonicity properties for the dynamics.
This is classical in the study of particle systems and ubiquitous in the literature (see \cite{Rost81} for an early reference). 
Then we present a proof of the identification of the spectral gap using the discrete Hopf-Cole transform. Finally we also treat in that section the particular case
of the TASEP (ASEP with $p=1$).

In Section \ref{sec:proofsth}, we present results (proved in subsequent sections) for the limiting behavior of the particle density (the hydrodynamic limit) and 
the positions of the leftmost particle and rightmost empty site.
Using these results and the materials on the spectral gap, we prove all our main theorems.

In Section \ref{Sec:Hydro} we prove the hydrodynamic result by extending the arguments developed in \cite{LabbeKPZ} to the constant biased case. 
Finally, in Section \ref{sec:lmprmes} we prove the results concerning convergence of the positions of the leftmost particle and the rightmost empty site by combining hydrodynamic limit 
estimates with coupling with a stationary variant of ASEP with a finite number of particles on the infinite line.

Sections \ref{sec:proofsth}, \ref{Sec:Hydro} and \ref{sec:lmprmes} are mostly independent and each of them can be read separately.

\section{Technical preliminaries}\label{sec:techos}

\subsection{Notation}

In some proofs, it is easier to deal with the following alternative notion of distance to equilibrium
\begin{equation}\label{alterdis} 
\bar d(t):= \max_{\xi, \xi' \in \gO} \| P^{\xi}_t-P^{\xi'}_t\|_{TV}.
\end{equation}
This does not raise any issue since the triangle inequality ensures that
$$ d(t)\le \bar d(t)\le 2 d(t).$$
We use the same superscript as for $d(t)$ when using the notation $\bar d(t)$ for one of the Markov chain introduced above.

\subsection{Height-function, monotone coupling}

A classical and convenient equivalent description of the particle system is given by the \textit{height function} obtained through the mapping 
$\eta \mapsto h(\eta)$ defined by 
$$\forall x\in \lint0, N\rint\;,\quad h(\eta)(x):=\sum_{y=1}^x \left(2\eta(y)-1\right) \;.$$
We let $\Omega_{N,k}:=h(\gO^0_{N,k})$ be the set of all such discrete height functions 
(which happen to be discrete bridges from $(0,0)$ to $(N,2k-N)$).
We denote by $(h^{\xi}(t,\cdot))_{t\geq 0}$ the dynamics on the height function: as $\eta\mapsto h(\eta)$ is injective, this contains the same information as the original dynamics and in particular is Markovian.

\begin{figure}\centering\label{Fig1}
	\begin{tikzpicture}[scale=0.3]
	
	\draw[->,thin,color=black] (0,0) node[below left] {$0$}-- (14,0) -- (14,0.2) -- (14,-0.2) -- (14,0) node[below] {$N$} -> (15,0);
	\draw[->,thin,color=black] (0,-9) -> (0,7);

	\draw[-,gray] (0,0) -- (6,6) -- (14,-2);
	\draw[-,gray] (0,0) -- (8,-8) -- (14,-2);
	
	\draw[-,thick,color=black] (0,0)  -- (1,1) -- (2,2) -- (3,1) -- (4,0) -- (5,-1) -- (6,0) -- (7,-1) -- (8,-2) -- (9,-1) -- (10,-2) -- (11,-1) -- (12,0) -- (13,-1) -- (14,-2);
	
	\draw[-,style=dotted] (4,0) -- (5,1) node[above] {\tiny \mbox{rate} $1-p$} -- (6,0);
	\draw[->] (5,-0.5) -- (5,0.5);
	
	\draw[-,style=dotted] (8,-2) -- (9,-3)node[below] {\tiny \mbox{rate} $p$} -- (10,-2);
	\draw[->] (9,-1.5) -- (9,-2.5);
	
	
	\end{tikzpicture}
	\caption{An example of height function with $k=6$ particles over $N=14$ sites. The interface lives within the grey rectangle.}
\end{figure}
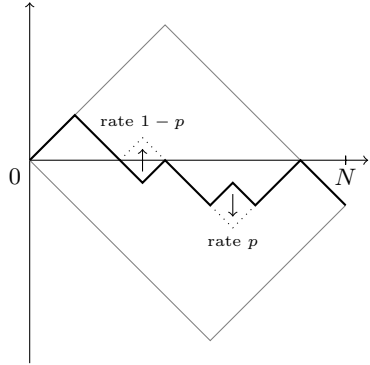

For the record, let us write the generator of the height-function dynamics. For $f: \gO_{N,k} \to \bbR$,
\begin{equation}\label{generatornk}
 \cL_{N,k}f(\xi):= \sum_{x=1}^{N-1} p\left[f(\xi_x)-f(\xi)\right]+(1-p)\left[f(\xi^x)-f(\xi)\right].
\end{equation}
where the configurations $\xi_x$ and $\xi^x$ are respectively defined by
\begin{equation*}
\begin{cases}
 \xi^x(y)=  \xi_x(y)=\xi(y) \quad \text{ for } y\in \llb 0,N\rrb \setminus \{x\}\;,\\
 \xi_x(x)= \max(\xi(x+1),\xi(x-1))-1\;,\\ 
 \xi^x(x)= \min(\xi(x+1),\xi(x-1))+1\;.
\end{cases}
\end{equation*}
In words, this simply means that local maxima (resp. minima) flip into local minima (resp. maxima) at rate $p$ (resp. $q$), see Figure \ref{Fig1}.

Similarly as in (\ref{projek}), we define a function $h_k$ that maps $\cS_N$ to $\gO_{N,k}$, by setting
\begin{equation}\label{defachk}
  h_k(\sigma)(x):= \sum_{i=1}^x \left(2\,\ind_{\{\sigma(i)\ge N-k+1\}}-1 \right)\;.
\end{equation}
Note that the knowledge of all the height-functions $h_k(\sigma)$, $k\in \lint 1,N-1\rint$ is sufficient to recover $\sigma$ completely:
\begin{equation}\label{permuheight}
\sigma(x)=N-k+1 \Leftrightarrow  \begin{cases}
h_k(\sigma)(x)-h_k(\sigma)(x-1)=1\;, \  \ \text{ and } \\

h_{k-1}(\sigma)(x)-h_{k-1}(\sigma)(x-1)=-1\;.\end{cases}
\end{equation}
Hence  $\sigma\mapsto  (h_k(\sigma))_{k=1}^{N-1}$ is injective.

The representation offers a very convenient framework to introduce an order on the state-space. The relation ``$\ge$'' is defined on $\gO_{N,k}\times \gO_{N,k}$, as follows
\begin{equation}\label{eq:orderheight}
\left\{ \ \xi_1\ge \xi_2 \ \right\} \Leftrightarrow  \left\{ \ \forall x\in \llb 0,N\rrb, \   \xi_1(x)\ge \xi_2(x) \right\}.
\end{equation}
The maximal element, which we denote by $\wedge$, corresponds to the configuration where all particles are on the left, and the minimal element 
$\vee$ corresponds to that where all particles are on the right. Even though this is not apparent in the notation, these two elements depend on $k$: we believe this will never raise any confusion in the sequel as the value $k$ will always be clear from the context.

\smallskip

These orders on the spaces $\gO_{N,k}$ induce an order on the group of permutations $\cS_N$ also denoted by ``$\ge$'', via the following relation
\begin{equation*}
 \left\{ \ \sigma_1 \ge \sigma_2 \ \right\}  \Leftrightarrow   \left\{ \ \forall k\in  \llb 1,N-1  \rrb,\  h_k(\sigma_1)\ge h_k(\sigma_2)\ \right\}
\end{equation*}
The minimal element is the identity $\Id$ and the maximal one is the permutation $\sigma_{\max}$ defined in
\eqref{defsigmamax}.

\smallskip

These orders are natural to consider since they are in some sense perserved by the dynamics.
It is indeed a classical result that we can construct a grand coupling that preserves the order in the following sense.

\begin{proposition}\label{orderpres}
There exists a coupling of the processes $(\sigma^{\xi}_t)_{t\ge 0}$ starting from $\xi\in\cS_N$
which satisfies
\begin{equation*}
\left\{ \ \xi\ge \xi' \ \right\} \Rightarrow  \left\{ \ \forall t \ge 0, \ \sigma^{\xi}_t\ge \sigma^{\xi'}_t \ \right\} .
\end{equation*}
For any $k\in \lint 1, N-1 \rint$ there exists a coupling of the processes  $(h^{\xi}(t,\cdot))_{t\ge 0}$ starting from $\xi\in\gO_{N,k}$ such that 
\begin{equation*}
\left\{ \ \xi\ge \xi' \ \right\} \Rightarrow  \left\{ \ \forall t \ge 0, \ h^{\xi}(t,\cdot)\ge h^{\xi'}(t,\cdot) \ \right\} .
\end{equation*}

\end{proposition}

We include a proof of this result in Appendix \ref{Appendix:Coupling} for completeness. 
Note that only the coupling for the card shuffling needs to be constructed since its projection \eqref{projek}
yields an order preserving coupling on $\gO_{N,k}$.
Dynamics with such order preserving property are usually called \textit{attractive}. In the rest of the paper, unless it is specified otherwise, we always work with such a grand coupling and use $\bbP$ to denote the associated probability 
distribution.

\subsection{Identification of the spectral gap}\label{sec:gap}

A tool which provides an intuition for identifying the spectral gap as well as the eigenfunctions of the generator
$\cL_{N,k}$ is the celebrated discrete Hopf-Cole transform, 
which was originally introduced by G\"artner~\cite{Gartner88} to derive the hydrodynamic limit of the process in a weakly asymmetric regime 
(when $p-1/2$ scales like $1/N$). 
Let us recall that the continuous Hopf-Cole transform $u\mapsto V(u)$ defined by $V(u)(t,x):=e^{c u(t,x)}$
allows to map the non-linear parabolic PDE
\begin{align*}
\partial_t u = \frac{1}{2}\partial^2_x u - \frac{c}{2} \left[1 - (\partial_x u)^2\right]\;,
\end{align*}
onto the linear parabolic PDE
\begin{align*}
\partial_t V = \frac12 \partial^2_x V - \frac{c^2}{2} V \;.
\end{align*}
Note that usually, the transformation which is considered is rather 
$V(u)(t,x):=e^{c [u(t,x)+\frac{c}{2} t]}$, but in our case we prefer not to have any time dependence in the expression.

Here, in analogy if one sets $u^{\xi}(t,x):=\bbE[h^{\xi}(t,x)]$, then it is not difficult to check from the expression \eqref{generatornk} of the generator $\cL_{N,k}$ that for all
$x\in\lint 1,N-1\rint$ we have
\begin{equation*}
 \partial_t u^{\xi}(t,x)=\frac{1}{2}\gD u^{\xi}(t,x)-(p-q)\bbE\big[\ind_{\{h^\xi(t,x+1)=h^\xi(t,x-1)\}}\big]\;,
\end{equation*}
where the discrete Laplace operator is defined by 
\begin{equation*}
\Delta f(x) = f(x+1)-2f(x)+f(x-1)\;,\quad x\in\lint 1,N-1\rint\;.
\end{equation*}
The second term is the discrete analogue of $[1-(\partial_x u)^2]$ ($\partial_x u$ being replaced by the mean slope on the interval $[x-1,x+1]$). This equation is non-linear in $u$.

\smallskip

In order to obtain a linear equation, we perform a discrete Hopf-Cole transform.
Due to discretization effects, $u\mapsto e^{ 2(p-q) u}$ is not the right transformation to consider. Additionally, we have to take care of the boundary conditions. We set
$$\zeta_x(\xi)=\gl^{\frac{1}{2}\xi(x)} \text{ and } \widetilde V(t,x)=\bbE \left[ \zeta_x(h^\xi(t,\cdot)) \right].$$
A computation yields that for all $x\in \lint 1,N-1\rint$,
\begin{equation*}
\cL_{N,k}(\zeta_x)(\xi) = \sqrt{pq}\, \gD (\lambda^{\frac12 \xi(x)}) - \varrho \lambda^{\frac12 \xi(x)}\;,
\end{equation*}
where
\begin{equation} \label{defrho}
\varrho:= (\sqrt{p}-\sqrt{q})^2\;.
\end{equation}
As $\sqrt{pq}\, \gD-\varrho$ is a linear operator on $\bbR^\bbZ$, it commutes with expectation. This immediately implies that for $x\in \lint 1,N-1\rint$
and $t\ge 0$ we have,
\begin{equation*}
\begin{cases}
	\partial_t \widetilde V(t,x)= (\sqrt{pq}\,\gD- \varrho) \widetilde V(t,x)\;,\quad x\in\llb 1,N-1  \rrb,\\
	\widetilde V(t,0)=1\;,\quad \widetilde V(t,N) =  \gl^{\frac{2k-N}{2}}\;.
\end{cases}
\end{equation*}
Finally, to deal with the boundary conditions, we let $a_{N,k}(x)$ be the solution of the following system
\begin{equation*}
 \begin{cases}
 	(\sqrt{pq}\, \gD-\varrho)a(x)=0\;,\quad x\in \llb 1,N-1  \rrb\;,\\
	 a(0)=1\;,\quad a(N)= \gl^{\frac{2k-N}{2}}  \;.
 \end{cases}
\end{equation*}
We can check that $V(t,x)=\tilde V(t,x)-a_{N,k}(x)$ satisfies for all $t\ge 0$ and $x\in \lint 0,N \rint$,
\begin{equation*}
\begin{cases}
	\partial_t V(t,x)= (\sqrt{pq}\gD- \varrho) V(t,x)\;,\quad x\in\llb 1,N-1  \rrb\;.\\
	V(t,0) = V(t,N) = 0\;.
\end{cases}
\end{equation*}
This equation allows to identify some eigenfunctions of $\cL_{N,k}$ by considering the decomposition of $V$ on a basis of eigenfunctions of $\sqrt{pq}\,\gD- \varrho$.
If one sets for $j=1,\dots,N-1$ 
\begin{equation*}
 f^{(j)}_{N,k}(\xi) := \sum_{x=1}^{N-1} \sin \left( \frac{xj\pi}{N} \right) \left( \gl^{\frac{1}{2}\xi(x)}- a_{N,k}(x) \right),
 \end{equation*}
then the projection of the equation at time zero on the $j$-th Fourier mode yields
\begin{equation*}
\cL_{N,k}  f^{(j)}_{N,k}(\xi) = \sqrt{pq} \sum_{x=1}^{N-1} \sin \left( \frac{xj\pi}{N} \right) \gD\left(\gl^{\frac{1}{2}\xi}- a_{N,k}\right)(x) 
-\varrho f^{(j)}_{N,k}(\xi).
\end{equation*}
Using discrete integration by parts twice (which do not yield boundary terms since both functions vanish at $0$ and $N$), and using the fact that $\sin \left( \frac{\cdot j\pi}{N} \right)$ is an eigenfunction for $\gD$, we obtain that
\begin{align*}
 \sum_{x=1}^{N-1} \gD\left(\sin \left( \frac{\cdot j\pi}{N} \right) \right)(x) \left(\gl^{\frac{1}{2}\xi(x)}- a_{N,k}(x)\right)
 =2\left(\cos \left(\frac{j\pi} N\right) - 1\right)  f^{(j)}_{N,k}(\xi)\;,
\end{align*}
and thus
\begin{equation}\label{Eq:GeneHopfCole}
 \cL_{N,k}  f^{(j)}_{N,k}=-(\varrho+ \gamma_{N}^{(j)}) f^{(j)}_{N,k}\;,
\end{equation}
where
\begin{equation*}
\gamma_{N}^{(j)}= 2\sqrt{pq}\left(1-\cos \left(\frac{j\pi} N\right) \right)=  4\sqrt{pq}\left[\sin\left(\frac{j\pi}{2N}\right)\right]^2.
\end{equation*}
Of course, except in the special cases $k=1$ or $k=N-1$, this is far from providing a complete basis of eigenfunctions (there are a total of $\binom{N}{k}$ of them), 
but this will turn out to be sufficient to identify the spectral gap.

First, and this is the most obvious part, considering the case $j=1$ (which minimizes the quantity $\varrho+ \gamma_{N}^{(j)}$) and setting $f_{N,k}:=f^{(1)}_{N,k}$ and $\gamma_{N}:=\gamma_{N}^{(1)}$, 
we obtain an upper bound on the spectral gap. This is valid for the ASEP, but also for the biased card shuffling as 
$f_{N,k}\circ h_k$ (recall \eqref{defachk}) is an eigenfunction for $\cL_N.$
Thus we have
\begin{equation*}
 \gap_{N,k}\le \varrho+ \gamma_{N} \text{ and } \gap_{N}\le \varrho+ \gamma_{N}.
\end{equation*}
To prove that this eigenvalue really corresponds to the spectral gap
an important observation is that $f_{N,k}$ is an increasing function on $\gO_{N,k}$ in the following sense  
$$\forall \xi, \xi' \in \gO_{N,k},\  \{\ \xi\le \xi'\ \}  \Rightarrow  \left\{ \ f_{N,k}(\xi) \le f_{N,k}(\xi')\ \right\}\;.$$
In \cite[Section 2.7]{Capetal12} it is shown that for a reversible attractive dynamics (see Proposition  \ref{orderpres}) with a maximal and a minimal element,
the eigenfunction corresponding to the spectral gap is increasing.
As it is quite difficult for two increasing functions to be orthogonal this indicates that $\varrho+ \gamma_{N}$ has to be the spectral gap.
We prove it in the next section by making use of the monotone coupling.
This is in fact a classical computation for attractive systems (see e.g. \cite[Section 3.1]{Wil04}) but we shall include it in full for the sake of completeness.

\subsection{Squeezing with monotone coupling}\label{sec:squizz}

Let us start with the case of ASEP with $k$ particles.
Recall that $\bbP$ is an order preserving coupling, and that $\vee$ and $\wedge$ denote the minimal and maximal configurations respectively.
Order-preserving implies in particular that once the dynamics starting from the two extremal height functions merge, the value of  $h^{\xi}(t,\cdot)$ is the same for all $\xi \in \gO_{N,k}$
\begin{equation}\label{merge}
 \left\{ \ h^{\vee}(t,\cdot) = h^{\wedge}(t,\cdot) \ \right\} \ \Rightarrow \  \left\{ \ \forall \xi \ne \xi',\ h^{\xi}(t,\cdot) = h^{\xi'}(t,\cdot) \ \right\}. 
\end{equation}
Using this, we consider $\xi,\xi' \in \gO_{N,k}$ which maximizes the total variation distance at time $t$ (recall \eqref{alterdis}) and argue as follows:
\begin{equation}\label{eq:couplaj}
\bar d_{N,k}(t)=\| P^{\xi}_t- P^{\xi'}_t \|_{TV}\le \P(h^{\xi}(t,\cdot) \ne h^{\xi'}(t,\cdot))
\le \P(h^\vee_t \ne h^\wedge_t ).
\end{equation}
Using the monotonicity of $f_{N,k}$, and the Markov inequality, we obtain that the quantity above is smaller than
\begin{equation}\label{ineq}
\P\left(f_{N,k}(h^\wedge_t) \ge  f_{N,k}(h^\vee_t)+ \delta_{\min}(f_{N,k})\right)
\le  \frac{\E[ f_{N,k}(h^\wedge_t)-f_{N,k}(h^\vee_t)]}{\delta_{\min}(f_{N,k})}\;,
\end{equation}
where
\begin{equation*}
\delta_{\min}(f_{N,k}) = \mintwo{\xi, \xi' \in \Omega_{N,k}}{\xi\ge \xi' \text{ and } \xi\ne \xi'} (f_{N,k}(\xi) - f_{N,k}(\xi')) \;.
\end{equation*}
By \eqref{Eq:GeneHopfCole}, we deduce that
\begin{equation}\label{eq:zioup}
d_{N,k}(t)\le \bar d_{N,k}(t) \le\frac{ \E[f_{N,k}(h^\wedge_t)-f(h^\vee_t)]}{\delta_{\min}(f_{N,k})}
= \frac{f_{N,k}(\wedge)-f_{N,k}(\vee)}{\delta_{\min}(f_{N,k})} e^{-(\gamma_N+\varrho) t}\;.
\end{equation}
In view of \eqref{symptogap} this implies that $\gap_{N,k}\ge (\gamma_N+\varrho)$ so that we conclude that 
\eqref{gapasep} holds.

\smallskip

Let us mention here that for all $k,N$,
\begin{equation}\label{Eq:BoundDelta}
\delta_{\min}(f_{N,k})=\min_{x\in \lint 1,N-1 \rint} \sin\left(\frac{x\pi}{N}\right) (\gl-1) \gl^{\vee(x)}\ge \frac{(\lambda-1)}{N} \lambda^{\frac{k-N}{2}}\;.
\end{equation}
Regarding \eqref{gapbias}, we observe as above that for all $\xi, \xi' \in \cS_N$ which maximize the total variation distance at time $t$, we have
\begin{equation}\label{eq:couplos}
\bar d_{N}(t)=\| Q^{\xi}_t- Q^{\xi'}_t \|_{TV}\le \P(\sigma^{\xi}_t \ne \sigma^{\xi'}_t).
\end{equation}
By injectivity of the height-function (recall \eqref{permuheight}) we have
\begin{equation}\label{eq:quantit}\begin{split}
\P(\sigma^{\xi}_t \ne \sigma^{\xi'}_t)&= \P(\exists k\in \llb 1,N-1\rrb, h_k(\sigma^{\xi}_t) \ne h_k(\sigma^{\xi'}_t ))
\\ &\le \sum_{k=1}^N \P(h_k(\sigma^{\xi}_t) \ne h_k(\sigma^{\xi'}_t)).
\end{split}\end{equation}
Then repeating \eqref{merge} and \eqref{eq:zioup}, it follows that each term in the last sum is smaller than 
$\left(f_{N,k}(\wedge)-f_{N,k}(\vee)\right) e^{-(\gamma_N+\varrho) t}$, which allows to conclude that \eqref{gapbias} holds.

To conclude let us remark that the above method provides us a quantitative upper-bound on the mixing time, but that it does not allow to identify the right constant.
More precisely (details are left to the reader), we have 
\begin{equation*}\begin{split}
 \bar d_{N,k}(t)\le C_\gl N \gl^{N/2}e^{-(\gamma_N+\varrho) t}\;,\quad \bar d(t)\le  C_\gl N^2  \gl^{N/2}e^{-(\gamma_N+\varrho) t}\;,
\end{split}\end{equation*}
for some constant $C_\gl>0$ depending on $\gl$. This gives an upper bound (which perhaps surprisingly does not depend on $k$) of order $ \frac{(\log \gl)}{2\varrho} N$ for both mixing times. 
While this is the right order of magnitude for the mixing time, the constant in front is not optimal. 
However, it can be remarked that for vanishing asymmetry, it gets asymptotically close to the right one for the case of biased shuffle or when $k=N/2$.

\subsection{The special case of $p=1$: TASEP}\label{sec:tasep}

Let us make some comments here about the case $p=1$ for which Theorem \ref{th:asep} is a simple consequence of the work of \cite{Rost81}. 
In that case, the system is mixed when $h^{\xi}(t,\cdot)$ hits the lowest configuration $\vee$ (in particular by monotonicity $\wedge$ is the worst configuration to start with).

Another remark is that $h^{\wedge}(t,\cdot)$ can be coupled with a TASEP dynamics on the infinite line with height function $h^{\wedge,\infty}(t,\cdot)$ 
 starting from 
$$\wedge_{\infty}(x):=x\ind_{\{x\le k\}}+ (2k-x)\ind_{\{x>k\}},$$
in such a way that for all $t$ and $x$
\begin{equation}\label{finiteinfinit}
 h^{\wedge}(t,\cdot):= \max(\vee(x),h^{\wedge,\infty}(t,\cdot)).
\end{equation}
Hence we have when $p=1$
\begin{equation*}
 d^{N,k}(t)= \bbP[ h^{\wedge}(t,\cdot) \ne \vee]= \bbP\left[ h^{\wedge,\infty}(t,N-k)>  -(N-k) \right].
\end{equation*}
Then it follows from \cite[Theorem 1]{Rost81} that when $k/N\to \alpha$, the limit is one if one chooses $t=N[(\sqrt{\alpha}+\sqrt{1-\alpha})^2-\gep]$ and 
 the limit is zero if $t=N[(\sqrt{\alpha}+\sqrt{1-\alpha})^2+\gep]$, thus yielding Theorem \ref{th:asep}.
 
 \smallskip

Since the work of Rost, much more detailed results have in fact been obtained about the scaling for $h^{\wedge,\infty}(t,\cdot)$ and its fluctuations are known 
\cite[Equation (3.7)]{Fer07} to be described by the $\Airy_2$ process. 
This information allows to deduce that when $k=N/2$, $d^{N,N/2}(t)$ drops from one to zero in a time window of order $N^{1/3}$ 
and even to identify the limit
of  $d^{N,N/2}\left(2N+N^{1/3}u\right),$ as a function of $u$ that can be expressed
in terms of the distribution of the $\Airy_2$ process.

An important thing to keep in mind is that a coupling such as \eqref{finiteinfinit} does not exist when $p<1$, the reason being that the boundary condition
have the effect of pushing $h(t,\cdot)$ in the upward direction. This is what makes the analysis of ASEP more difficult, and the main reason why
the question of cutoff has been open for more than a decade.
While we do believe that the statement about the $N^{1/3}$ cutoff window and profile should remain valid when $p\in(0,1)$, and also for every $\alpha\in(0,1)$, 
they remain at this stage, challenging conjectures.

\section{Getting the mixing times from scaling limits}\label{sec:proofsth}

In the present section, we show how to reduce the proof of the mixing time for both processes to a scaling limit statement about the positions of the leftmost 
particle and the rightmost empty site (Proposition \ref{th:fronteerlimit}). The underlying idea is that the contraction inequality derived from the eigenfunction $f_{N,k}$, while not providing 
a sharp estimate on the mixing time, allows to prove that once
the system is macroscopically close to equilibrium, it mixes rapidly.

\subsection{The hydrodynamic profile}\label{sec:hydropro}

Let $\xi_N$ be a sequence of elements of $\cup_k \Omega_{N,k}^0$ and let us define the associated sequence of empirical densities
$$ \rho^N_t(dy) = \frac1{N}\sum_{x=1}^N \eta^{\xi_N}\left(\frac{Nt}{p-q},x\right) \delta_{x/N}(dy)\;.$$
Notice that $\rho^N_t$ belongs to the convex set $\cM$ of measures on $[0,1]$ with total-mass at most $1$. We endow this set with the topology of weak convergence, see \eqref{Eq:WeakCV}.\\

We assume that $\rho^N_0$ converges weakly to some limiting density $\rho_0$: this is a harmless assumption since the sequence $\rho^N_0$ is tight. As $\langle \rho^N_0,\varphi\rangle \leq \frac1{N}\sum_{x=1}^N \varphi(x/N)$ for any $\varphi\in \cC([0,1],\bbR_+)$, it is simple to check that $\rho_0$ necessarily belongs to the dual of $L^1$, namely to $L^\infty$, and satisfies $\rho_0(x)\in [0,1]$ for almost every $x\in [0,1]$. We also let $h^{\xi_N}(t,x)$ be the associated height function, and we define $u^N(t,x) = \frac1{N}h(\frac{N}{p-q}t,xN)$ for all $t\geq 0$ and $x\in[0,1]$.

\begin{theorem}\label{Th:HydroGene}
The sequence $\rho^N$ converges in distribution in the Skorohod space $\bbD(\bbR_+,\cM)$ towards the unique entropy solution of the 
inviscid Burgers equation with zero-flux boundary conditions:
\begin{align}\label{Eq:BurgersZero}
\begin{cases}
\partial_t \rho = -\partial_x \big( \rho(1-\rho) \big)\;,\qquad t>0\;, x\in(0,1)\;,\\
\rho(t,x)(1-\rho(t,x)) = 0\;,\qquad  \quad t>0\;,x\in\{0,1\}\;,\\
\rho(0,\cdot) = \rho_0(\cdot)\;.
\end{cases}
\end{align}
Furthermore, the sequence $u^N$ converges in $\bbD(\bbR_+,\cC([0,1]))$ towards the integrated solution $u(t,x) = \int_0^x (2\rho(t,y)-1)\dd y$.
\end{theorem}
The Skorohod spaces denote the spaces of cadlag functions and are endowed with the Skorohod topology, see Billingsley~\cite{Billingsley}.

\begin{remark}
 Note that intuitively $u$ should be the solution of the following equation 
 \begin{align}\label{Eq:Burgers1}
\begin{cases}
\partial_t u = -\frac{1}{2}(1-(\partial_x u)^2)\;,\qquad t>0\;, x\in(0,1)\;,\\
u(0,\cdot) = 0 \quad  \text{and} \quad u(1,\cdot)=2\alpha-1\;, \qquad  t>0\;,x\in\{0,1\}\;,\\
u(0,\cdot) = \int_0^{\cdot} (2\rho(t,y)-1)\dd y \;.
\end{cases}
\end{align}
However, the precise connection between \eqref{Eq:BurgersZero} and \eqref{Eq:Burgers1} has not been established in the literature. So we stick to the problem \eqref{Eq:BurgersZero} formulated in terms of particle density, as it is sufficient to our purpose.
\end{remark}

\begin{remark}
All the solutions of \eqref{Eq:BurgersZero} obtained in Theorem \ref{Th:HydroGene} stabilize in finite time to an equilibrium profile, given by 
 $\ind_{[1-\alpha,1]}$ where $\alpha=\int_{[0,1]}\rho_0(x)\dd x$. Indeed, the explicit solution starting from $\wedge$ stabilizes in finite time and stays above any other solution, by monotonicity of the particle system. We define thus 
 \begin{equation}\label{def:trho}
  t_{\rho_0}:=\inf\{ t>0\ | \ \rho(t,\cdot)=  \ind_{[1-\alpha,1]} \}.
 \end{equation}
This is the quantity involved in the expression of the mixing time in Theorem \ref{th:asep2}.
Note that while in the extremal case described in Theorem \ref{th:asep}, the mixing time coincides  
with $N(p-q)^{-1}t_{\rho_0}$ for $\rho_0:=\ind_{[0,\alpha]}$, this is not always the case when starting from an arbitrary condition as the position of 
the leftmost particle and rightmost empty site can, in some cases, take a longer time to reach equilibrium than $\rho$.
\end{remark}

The precise definition of the entropy solutions of (\ref{Eq:BurgersZero}) as well as the proof of the convergence is postponed to Section \ref{Sec:Hydro}.

Let us describe the scaling limit of the height function starting from the maximal element $\wedge$. This object is relevant only when the density of particle $\alpha$ is strictly positive: at the end of the present subsection, we introduce the relevant quantities when $\alpha=0$. For $\alpha\in (0,1/2]$, we define $\vee_{\alpha} : [0,1] \to \bbR$, $\wedge_{\alpha} : [0,1] \to \bbR$ which correspond to the extremal macroscopic states, 
\begin{equation*}
\vee_{\alpha}(x):= \max(-x,x-2(1-\alpha))\;,\qquad \wedge_{\alpha}(x):= \min(x,2\alpha-x)\;,
\end{equation*}
and let $g_{\alpha}: \bbR_+\times [0,1]\to \bbR$ be defined as follows
\begin{equation*}\begin{split}
g^0_{\alpha}(t,x)&:= \begin{cases} \alpha-\frac{t}{2}-\frac{(x-\alpha)^2}{2t}, \quad &\text{ if } |x-\alpha| \le t,  \\
\wedge^{\alpha}(x), \quad & \text{ if } |x-\alpha| \ge t,
\end{cases}\\
g_{\alpha}(t,x)&:= \max(\vee_{\alpha}(x), g^0_{\alpha}(x,t)).
\end{split}
\end{equation*}

As a consequence of Theorem \ref{Th:HydroGene}, we obtain the following statement.

\begin{corollary}\label{th:hydrosimples}
 
Let $p\in (1/2,1]$ and $\alpha \in (0,1/2]$. For any $\gep>0$ and any $T>0$, we have
 
 \begin{equation*}
  \lim_{N\to \infty} \bbP \left[ \sup_{t\in [0,T]}\sup_{x\in [0,1]} \left|\frac1{N}h^{\wedge}\left(\frac{Nt}{p-q},Nx \right)- g_{\alpha}(t,x)\right|\le \gep \right]=0\;.
 \end{equation*}

 \end{corollary}
 
Let us introduce $\ell_\alpha(t), r_{\alpha}(t)\in [0,1]$, which for $t\le (\sqrt{\alpha}+\sqrt{1-\alpha})^2$ are the extremities 
of the interval on which  $g_\alpha(t,\cdot)$ and $g^0_{\alpha}(t,x)$ coincide):

\begin{equation*}\begin{split}
\ell_\alpha(t)&=\begin{cases} 0 \quad &\text{ if } t\le \alpha\;,\\ 
                      (\sqrt{t}-\sqrt{\alpha})^2 \quad &\text{ if } t\in \left(\alpha, (\sqrt{\alpha}+\sqrt{1-\alpha})^2 \right)\;,\\
                      1-\alpha \quad &\text{ if } t\ge   (\sqrt{\alpha}+\sqrt{1-\alpha})^2\;,
\end{cases}
\end{split}
\end{equation*}
and
\begin{equation*}
\begin{split}
r_{\alpha}(t)&=\begin{cases} 1 \quad &\text{ if } t\le 1-\alpha\;,\\ 
                      1-(\sqrt{t}-\sqrt{1-\alpha})^2 \quad &\text{ if } t\in \left(1-\alpha, (\sqrt{\alpha}+\sqrt{1-\alpha})^2 \right)\;,\\
                      1-\alpha \quad &\text{ if } t\ge   (\sqrt{\alpha}+\sqrt{1-\alpha})^2\;.
\end{cases}\end{split}
 \end{equation*}
 
When $\alpha = 0$, the hydrodynamic limit does not evolve since 
the system is macroscopically at equilibrium at time $0$ even though it is far from the microscopical equilibrium. We introduce:
$$ \ell_0(t)=  t\wedge 1\quad \text{ and }\quad r_0(t)=1,\qquad t\ge 0\;.$$

\subsection{Scaling limit for rightmost particle and leftmost empty site}

We introduce now the key statements that will allow us to get sharp mixing time estimates: given $\xi\in \gO_{N,k}$, we let $\ell_{N,k}(\xi)$ and $r_{N,k}(\xi)$ 
denote the position of the leftmost particle and the rightmost empty site respectively (which almost corresponds to the quantities introduced in \eqref{def1:lknrkn} for $\xi\in  \gO^{0}_{N,k}$). In terms of height function this translates into
\begin{equation}\label{def:lknrkn}
\begin{split}
\ell_{N,k}(\xi)&=\max\{x\in\llb 0, N \rrb \ : \  \xi(x)= -x \},\\
r_{N,k}(\xi)&=\min\{x\in\llb 0, N \rrb  \ : \  \xi(x)= x-2(N-k) \}.
\end{split}
\end{equation}
We also set 
\begin{equation}\label{def:LLRR}
L_{N,k}(t):=\ell_{N,k}(h^{\wedge}_t)  \quad \text{ and }  \quad R_{N,k}(t):=r_{N,k}(h^{\wedge}_t). 
\end{equation}
Our result is the following.
 \begin{proposition}\label{th:fronteerlimit}
Let $p\in (1/2,1]$ and $\alpha \in [0,1/2]$. For any $t\geq 0$ we have the following convergences in probability
 \begin{equation*}
  \limtwo{N\to \infty}{k/N\to \alpha} N^{-1}L_{N,k}\left( \frac{Nt}{p-q}  \right)=\ell_{\alpha}(t)\; \text{ and }
  \limtwo{N\to \infty}{k/N\to \alpha} N^{-1} R_{N,k}\left( \frac{Nt}{p-q}  \right)=r_{\alpha}(t)\;.
  \end{equation*}
\end{proposition}

Let us stress that in the case where $\alpha \in (0,1/2]$, Corollary \ref{th:hydrosimples} does not directly imply Proposition \ref{th:fronteerlimit}. It rather states that $\ell_{\alpha}(t)$ is the smallest point where the particle density is non-zero and $r_\alpha(t)$ is the largest point where it is not equal to $1$ so that, provided the limit exits in probability, 
one deduces the upper bound for $L_{N,k}$ and the lower bound for $R_{N,k}$:
 \begin{equation*}
  \limtwo{N\to \infty}{k/N\to \alpha} N^{-1}L_{N,k}\left( \frac{Nt}{p-q}  \right)\leq\ell_{\alpha}(t)\; \text{ and }
  \limtwo{N\to \infty}{k/N\to \alpha} N^{-1} R_{N,k}\left( \frac{Nt}{p-q}  \right)\geq r_{\alpha}(t)\;.
\end{equation*}
However it does not give the microscopic information that would be necessary to obtain the lower bound for $L_{N,k}$ and the upper bound for $R_{N,k}$. 
Another important observation is that this microscopic information is not contained in our proof of Theorem \ref{Th:HydroGene}:
the technique we use to derive $g_{\alpha}(t,x)$ as a scaling limit, is to show that the particle density is an entropy solution of the inviscid Burgers equation,
and this formulation of the problem does not allow to track the positions of the leftmost particle and the rightmost empty site.

\noindent Let us now introduce
\begin{equation*}
A^{\gep}_{N,k}:= \{ \xi \in \gO_{N,k} \ : \ |\ell_{N,k}(\xi)-N+k|\le \gep N   \text{ and } |r_{N,k}(\xi)-N+k|\le \gep N \}\;,
\end{equation*}
and 
$$t_{\alpha,N}:= \frac{N}{p-q}(\sqrt{\alpha}+\sqrt{1-\alpha})^2\;.$$
As a direct consequence of Proposition \ref{th:fronteerlimit}, we get the following result.

\begin{corollary}\label{cor:hit}
Let $p\in (1/2,1]$ and $\alpha \in [0,1/2]$. We have for any $\delta\ge 0$ and $\gep>0$
 \begin{equation}\label{eq:apres}
  \limtwo{N\to \infty}{k/N\to \alpha}\bbP \left( h^{\wedge}_{(1+\delta)t_{\alpha,N}} \in A^{\gep}_{N,k} \right)=1.
\end{equation}
Moreover for any fixed $\delta>0$, for $\gep \le \gep_0(\delta)$ we have 
 \begin{equation}\label{eq:avant}
  \limtwo{N\to \infty}{k/N\to \alpha}\bbP \left( h^{\wedge}_{(1-\delta)t_{\alpha,N}} \in A^{\gep}_{N,k} \right)=0.
\end{equation}
 
\end{corollary}

Finally in order to use these results, we need to check that the final positions of $\ell_{\alpha}$, $r_{\alpha}$ correspond indeed to equilibrium.
This is a known estimate but we include a proof for completeness.

\begin{lemma}\label{lem:localise}
We have for all values of $N$ and $k$
\begin{equation*}
 \pi_{N,k}\left( \big|\ell_{N,k}-r_{N,k}\big|\ge M\right)\le  \frac{\gl^{3-M}(M+1)}{(\gl-1)^2}.
\end{equation*}
\end{lemma}
\begin{proof}
As in \cite[Proof of Proposition 11]{LevPer16}, we rely on the observation that if $\ell_{N,k}(\xi)=i-1, r_{N,k}(\xi)=j$ (not both equal to $N-k$), then there is a particle at site $i$ while site $j$ is empty, and therefore, the bijection $T_{i,j}$ that interchanges the contents of sites $i$ and $j$ in the particle system, decreases $A(\xi)$ (recall \eqref{eq:equilibro}) by an amount $j-i$. Thus we have  
\begin{equation*}
\pi_{N,k}\left( \ell_{N,k}=i-1, r_{N,k}=j\right)\le \gl^{-(j-i)} \sum_{\xi \in \gO_{N,k}}\pi_{N,k}(T_{i,j}(\xi))\le  \gl^{-(j-i)},
\end{equation*}
Summing over all possibilities for $i$ and $j$ (at most $A$ possibilities when the gap is equal to $A$) this yields 
\begin{equation*}
\pi_{N,k}\left( \big|\ell_{N,k}-r_{N,k}\big|\ge M\right)\le \sum_{A\ge M} A \gl^{-(A-1)}= \frac{M\gl^{-(M-1)}}{(1-\gl^{-1})}+\frac{\gl^{-M}}{(1-\gl^{-1})^2}.
 \end{equation*}
\end{proof}
 
\subsection{The mixing time for ASEP: Proof of Theorem \ref{th:asep}}\label{Subsec:MixingAsep}
 
As the case $p=1$ was treated in Section \ref{sec:tasep}, we assume here that $p\in(1/2,1)$.
 We have to prove that in the limit when $k/N$ tends to $\alpha$, the mixing time is equivalent to $t_{\alpha,N}$.
 
 \smallskip

The easiest part is to show that the mixing time is \textit{at least}  $t_{\alpha,N}$, or more precisely, that for any $\delta>0$
 \begin{equation*}
  \limtwo{N\to \infty}{k/N\to \alpha}d_{N,k}((1-\delta)t_{\alpha,N})=1.
 \end{equation*}
From the definitions of total variation distance and $d_{N,k}$, we have for any $\gep>0$
\begin{equation*}
d_{N,k}((1-\delta)t_{\alpha,N})\ge \pi_{N,k}(A^{\gep}_{N,k})-P^{\wedge}_{(1-\delta)t_{\alpha,N}}(A^{\gep}_{N,k}).
\end{equation*}
The first probability converges to $1$ according to Lemma \ref{lem:localise}, while the second converges to zero
if $\gep$ is sufficiently small according to Corollary \ref{cor:hit}.

\smallskip

To obtain the other bound on the mixing time, we show that for any $\delta>0$,
 \begin{equation}\label{eq:bornesup}
  \limtwo{N\to \infty}{k/N\to \alpha}d_{N,k}((1+\delta)t_{\alpha,N})=0.
 \end{equation}
Recall \eqref{eq:couplaj} and the monotone grand coupling. We have 
 \begin{equation*}
  d^{N,k}((1+\delta)t_{\alpha,N})\le \bbP\left[ h^{\wedge}_{(1+\delta)t_{\alpha,N}}\ne h^{\vee}_{(1+\delta)t_{\alpha,N}} \right].
 \end{equation*}

The right-hand side is smaller than 
\begin{equation*}
     \bbP\left[ h^{\wedge}_{(1+\delta)t_{\alpha,N}}\ne h^{\vee}_{(1+\delta)t_{\alpha,N}} \ | \ h^{\wedge}_{t_{\alpha,N}}\in A^{\gep}_{N,k} \right]
     + \bbP\left[ h^{\wedge}_{t_{\alpha,N}}\notin A^{\gep}_{N,k} \right].
\end{equation*}
According to \eqref{eq:apres}, the second term vanishes in the limit. Regarding the first term, using the Markov property and repeating the computation of Section \ref{sec:squizz}, we obtain
\begin{equation}\label{eq:condisquizz}\begin{split}
&\bbP\left[  h^{\wedge}_{(1+\delta)t_{\alpha,N}}\ne h^{\vee}_{(1+\delta)t_{\alpha,N}} \ | \ h^{\wedge}_{t_{\alpha,N}},\ h^{\vee}_{t_{\alpha,N}} \right] \\
&\le \bbE\left[ \frac{f_{N,k}\left(h^{\wedge}_{(1+\delta)t_{\alpha,N}}\right)-f_{N,k}\left( h^{\vee}_{(1+\delta)t_{\alpha,N}}\right)}{\delta_{\min}(f_{N,k})} \ | 
\  h^{\wedge}_{t_{\alpha,N}},\ h^{\vee}_{t_{\alpha,N}} \right]\\
&\le \frac{f_{N,k}\left(h^{\wedge}_{t_{\alpha,N}}\right)-f_{N,k}\left( h^{\vee}_{t_{\alpha,N}}\right)}{\delta_{\min}(f_{N,k})}e^{-\delta t_{\alpha,N}(\varrho+\gamma_N)}.\end{split}
\end{equation}
To conclude we remark that if one defines  
$\wedge^{\gep}_{N,k}$ to be the maximal element of $A^{\gep}_{N,k}$ (which is well defined since the maximum of two elements 
of $A^{\gep}_{N,k}$ is in $A^{\gep}_{N,k}$), then for some constant $C=C_\gl>0$, we have on the event $\{h^{\wedge}_{t_{\alpha,N}}\in A^{\gep}_{N,k}\}$
\begin{equation*}
 f_{N,k}\left(h^{\wedge}_{t_{\alpha,N}}\right)-f_{N,k}\left( h^{\vee}_{t_{\alpha,N}}\right) \le  f_{N,k}\left( \wedge^{\gep}_{N,k}\right)-f_{N,k}(\vee)
 \le  CN \delta_{\min}(f_{N,k}) \gl^{\gep N}\;,
\end{equation*}
where we have used \eqref{Eq:BoundDelta}. Hence we deduce from \eqref{eq:condisquizz} that 
\begin{equation*}
   \bbP\left[ h^{\wedge}_{(1+\delta)t_{\alpha,N}}\ne h^{\vee}_{(1+\delta)t_{\alpha,N}} \ | \ h^{\wedge}_{t_{\alpha,N}}\in A^{\gep}_{N,k} \right]
   \le CN \gl^{\gep N}e^{-\delta t_{\alpha,N}(\varrho+\gamma_N)}.
\end{equation*}
If $\gep$ is chosen small compared to $\delta$, this last term tends to zero exponentially fast and we can conclude that \eqref{eq:bornesup} holds.

\subsection{From ASEP to card-shuffle: Proof of Theorem \ref{th:shuffle}}

Let $\sigma^{\max}_t$ and $\sigma^{\min}_t$ denote the dynamics starting from the maximal ($\sigma^{\max}$ from \eqref{defsigmamax})
and minimal (the identity) initial conditions.
In order to adapt the method used above for the ASEP we need to control the height functions for all levels $k$.
We extend the definition \eqref{def:lknrkn} to the group of permutations by setting for $\xi\in \cS_N$ 
\begin{equation*}
 \ell_{N,k}(\xi):=\ell_{N,k}(h_k(\xi))  \quad \text{ and }  \quad r_{N,k}(\xi):=r_{N,k}(h_k(\xi)). 
\end{equation*}
We also set  
\begin{align*}
 B^{\gep}_{N}&:=\Big\{\xi \in \cS_N : \forall k\in \llb 1, N-1 \rrb,\\
 &\qquad\qquad\qquad |\ell_{N,k}(\xi)-N+k|\le \gep N, |r_{N,k}(\xi)-N+k|\le \gep N\Big\}\;.\\
 &=\Big\{\xi \in \cS_N : \forall k\in \llb 1, N-1 \rrb,\; h_k(\xi) \in A^\gep_{N,k}\Big\}\;.
\end{align*}
We need the following improvement of Corollary \ref{cor:hit}.
We set 
$$t_N:=t_{1/2,N}=2(p-q)^{-1}N$$
\begin{lemma}\label{lem:intheb}
 We have for any $\gep>0$
 \begin{equation*}
  \lim_{N\to \infty} \bbP\left[\sigma^{\max}_{t_N} \in  B^{\gep}_{N} \right]=1.
 \end{equation*}

\end{lemma}

\begin{proof}
 
Let $m>0$ be a fixed integer and fix $k^N_1<\dots<k^N_m$ for $N\ge 0$, such that
\begin{equation*}
\forall i \in \llb 1,m\rrb, \lim_{N\to \infty} k_i^N/N= \frac{i}{m+1},
\end{equation*}
Set 
\begin{align*}
B^{\gep}_{N,m}&:=\big\{\xi \in \cS_N :\forall i\in \llb 1, m \rrb,\\
&|\ell_{N,k_i}(\xi)-N+k_i|\le (\gep/2) N \text{ and } |r_{N,k_i}(\xi)-N+k_i|\le (\gep/2)N)\big\}.
\end{align*}
For a given $\xi\in \cS_N$, $\ell_{N,k}(\xi)$ and $r_{N,k}(\xi)$ are non-increasing functions of $k$ (when $k$ increases, only new particles are added).
Thus, we have for $m\ge 3\gep^{-1}$ and $N$ large enough $B^{\gep}_{N,m}\subset B^{\gep}_N$. Consequently 
\begin{align*}
  \bbP\left[\sigma^{\max}_{t_N} \notin  B^{\gep}_{N} \right]\le\bbP\left[\sigma^{\max}_{t_N} \notin  B^{\gep}_{N,m} \right] &\le 
 \sum_{i=1}^m \bbP\left[ h_k(\sigma^{\max}_{t_N})\notin A^{\gep/2}_{N,k_i} \right]\\
 &= \sum_{i=1}^m \bbP\left[ h^{\wedge}_{t_N}\notin A^{\gep/2}_{N,k_i} \right]. 
 \end{align*}
 Using Corollary \ref{cor:hit}, and the fact that $t_N=\max_{\alpha\in[0,1]} t_{N,\alpha}$, we can conclude. 
\end{proof}

\begin{proof}[Proof of Theorem \ref{th:shuffle}]
Note that we only need to prove an upper-bound since the lower bound is given by the case $\alpha=1/2$ of Theorem \ref{th:asep}. 
We consider separately the case $p=1$ in the next subsection. Recall \eqref{eq:couplos}. We have for all $\delta > 0$
\begin{equation*}
  d_N((1+\delta)t_N)\le \bbP\left[  \sigma^{\max}_{(1+\delta)t_N} \ne \sigma^{\min}_{(1+\delta)t_N} \ | \ \sigma^{\max}_{t_N}\in \  B^{\gep}_{N} \right]
  +\bbP\left[  \sigma^{\max}_{t_N}\notin \  B^{\gep}_{N}\right].
\end{equation*}
From Lemma \ref{lem:intheb}, the second term goes to zero.
To estimate the first one, we use a conditional version of \eqref{eq:quantit} and \eqref{eq:zioup} and we obtain 
\begin{equation}\label{eq:squizz}\begin{split}
  &\bbP\left[  \sigma^{\max}_{(1+\delta)t_N} \ne \sigma^{\min}_{(1+\delta)t_N} \ | \ \sigma^{\max}_{t_N}, \ \sigma^{\min}_{t_N} \right]\\
  &\le \sum_{k=1}^N \frac{f_{N,k}(h_k(\sigma^{\max}_{t_N}))-f_{N,k}(h_k(\sigma^{\min}_{t_N}))}{\delta_{\min}(f_{N,k})}e^{-\delta t_N(\varrho+\gamma_N)}.
\end{split}\end{equation}
Now if $\sigma^{\max}_{t_N}$ is in $B^{\gep}_N$ then all its projections are in the respective $A_{N,k}^{\gep}$ and thus (recall that 
$\wedge^{\gep}_{N,k}$ is the maximal element of $A^{\gep}_{N,k}$)

\begin{equation*}
 f_{N,k}(h_k(\sigma^{\max}_{t_N}))-f_{N,k}(h_k(\sigma^{\min}_{t_N}))\le  f_{N,k}(\wedge^{\gep}_{N,k})-f_{N,k}(\vee)
 \le CN \delta_{\min}(f_{N,k})\gl^{\gep N}.
\end{equation*}
Thus taking the conditional expectation and applying Lemma \ref{lem:intheb} we deduce that there exists $C'>0$ such that for all $N$ large enough
\begin{equation*}
  \bbP\left[  \sigma^{\max}_{(1+\delta)t_N} \ne \sigma^{\min}_{(1+\delta)t_N} \ | \ \sigma^{\max}_{t_N}\in B^{\gep}_N \right]
  \le C'N^2 \gl^{\gep N}e^{-\delta t_N(\varrho+\gamma_N)}.
\end{equation*}
The right hand side tends to zero exponentially fast provided $\gep$ is chosen sufficiently small compared to $\delta$. Therefore, we have shown that $d_N((1+\delta)t_N)$ goes to $0$ as $N\rightarrow\infty$, thus concluding the proof of Theorem \ref{th:shuffle} in the case $p<1$.
\end{proof}

\subsection{The case $p=1$ for Theorem \ref{th:shuffle}}\label{Subsec:p1}

Unlike for the TASEP, the mixing for the totally biased card shuffling cannot be obtained directly from \cite[Theorem 1]{Rost81}.
The reason being that for doing so one would need to know not only the limiting behavior of the hitting time of $\vee_{N,k}$, but also some estimates on the rate of 
convergence.
While this could be achieved by using large deviation results obtained for the TASEP (see e.g.\ \cite{DerLeb98}), 
we prefer in this section to show how the case $p=1$ can be deduced from $p<1$ via approximations.
We proceed in two steps.

\smallskip

First, given $\gep$ sufficiently small (independent of $N$), we let $(\sigma_{t})_{t\ge 0}$ and $(\tilde \sigma_t)_{t\ge 0}$ be biased card shuffles 
with respective asymmetries given by $p=1$, and 
$\tilde p=1-\epsilon/2$, both with initial condition $\sigma^{\max}$. 
We denote by $\tilde \pi_{N}$
the equilibrium measure associated to $\tilde \sigma$. We couple these two processes (using a construction similar to the one displayed in Appendix \ref{Appendix:Coupling}) in such a way that for all $t\ge 0$,
\begin{equation}\label{dominacion}
\tilde \sigma_t\ge \sigma_t.
\end{equation}
Notice that for $N$ sufficiently large, Theorem \ref{th:shuffle} for $\tilde{p}$ yields that $t_1:=2N(1+2\gep)$ 
is larger than the $\gep$-mixing time of $\tilde \sigma_t$. We introduce
\begin{multline*}
\tilde{B}_{N}:= \{ \xi \in \cS_{N} \ : \forall k\in \lint 1, N\rint,\\
\  \ell_{N,k}(\xi)\ge N-k-\sqrt{N}   \text{ and } r_{N,k}(\xi)\le N-k+\sqrt{N} \}\;.
\end{multline*}
As $\tilde{B}_{N}$ is a decreasing event, for $N$ large enough we have from \eqref{dominacion} and the mixing estimate for $\tilde{p}$ 
\begin{equation}\label{secondterm}
 \bbP\left[ \sigma_{t_1} \notin \tilde{B}_{N} \right]\le  \bbP\left[ \tilde\sigma_{t_1} \notin \tilde{B}_{N} \right]
 \le \tilde{\pi}_{N}(\cS_N\backslash\tilde{B}_{N})+\gep \le 2\gep.
\end{equation}
where the last estimate is obtained using Lemma \ref{lem:localise} and a union bound over $k$.

\smallskip

We use then a second coupling for $t\ge t_1$: Let $(\hat \sigma_{t})_{t\ge t_1}$ be a biased shuffling with asymmetry $\hat p=1-N^{-2}$ 
 and initial condition $\sigma_{t_1}$, and coupled with $(\sigma_t)_{t\ge t_1}$ in a way such that at all time $\hat \sigma_t\ge \sigma^{\max}_t$.
Also, we let $\hat \sigma^{\min}_{t}$ be a shuffling with bias $\hat p$ starting from the identity at time $t_1$,
and couple it in a way such that $\hat \sigma_t\ge \hat \sigma^{\min}_t$. We set $t_2=t_1+\gep N$, and we denote by $\hat{\pi}_N$ the equilibrium measure associated to $\hat{\sigma}$. We have 
\begin{equation*}
d_N(t_2)=\bbP\left[ \sigma_{t_2}\ne \Id \right]\le \bbP\left[ \hat \sigma_{t_2}\ne \Id \right]
\le  \bbP\left[ \hat \sigma_{t_2}\ne \Id \ | \ \sigma_{t_1}\in \tilde{B}_{N} \right]+  \bbP\left[ \sigma_{t_1} \notin \tilde{B}_{N} \right].
\end{equation*}
The second term is smaller than $2\gep$ from \eqref{secondterm}.
The first term can be decomposed as follows
\begin{equation*}
 \bbP\left[ \hat \sigma_{t_2}\ne \Id \ | \ \sigma_{t_1}\in \tilde B_N\right]
 \le  \bbP\left[ \hat \sigma_{t_2}\ne \hat \sigma^{\min}_{t_2} \ | \ \sigma_{t_1}\in \tilde B_N \right]
 +\bbP\left[ \hat \sigma^{\min}_{t_2}\ne \Id \right].
\end{equation*}
Using a stochastic coupling with equilibrium we see that there exists $C>0$ such that
$$ \bbP\left[ \hat \sigma^{\min}_{t_2}\ne \Id \right] \le \hat \pi_N(\cS_N\setminus \{\Id \})\le CN^{-1}.$$
where the last estimate can be deduced from Lemma \ref{lem:localise}.
The other contribution is bounded using the squeezing argument in \eqref{eq:squizz}: 
if $\hat f_{N,k}$ are the eigenfunctions corresponding to $\hat p$ we have 
\begin{equation*}
 \bbP\left[ \hat \sigma_{t_2}\ne \hat \sigma^{\min}_{t_2} \ | \ \sigma_{t_1} \right]\le
  \left( \sum_{k=1}^N \frac{\hat f_{N,k}(h_k(\sigma_{t_1}))-\hat f_{N,k}(h_k(\Id))}{\delta_{\min}(\hat f_{N,k})}\right)e^{-\gep N (\hat \varrho+\gamma_N)}.
\end{equation*}
If $\sigma_{t_1}\in \tilde{B}_{N}$, then the first factor in the r.h.s.\ is bounded above by $N^{C\sqrt{N}}$ for some constant $C$, while the second term
is smaller than $e^{-\gep N/2}$. This allows to conclude the proof.
\qed

\subsection{Proof of Theorem \ref{th:asep2}}

In this subsection, we present the modifications needed to obtain the mixing time starting from some general sequence of initial 
conditions $(\xi_N,N\geq 1)$ satisfying \eqref{scalingz}. We focus here on the case $\alpha>0$ since $\alpha=0$ is much simpler and can be 
immediately adapted from the material in Section \ref{sec:zeralfa}.
We denote by $(h_t^{\xi_N},t\geq 0)$ the associated evolving height function, and we update the notation \eqref{def:LLRR} to fit the initial condition
$$L_{N,k}(t) := \ell_{N,k}(h_t^{\xi_N}) \text{ and } R_{N,k}(t) := r_{N,k}(h_t^{\xi_N})$$ 
Let $(\rho(t,x),t\geq 0, x\in (0,1))$ be the hydrodynamic limit obtained in Theorem \ref{Th:HydroGene}. 
Recall the definition of $t_{\rho_0}$ \eqref{def:trho}, and set
\begin{equation*}
 \ell_{\rho}(t):=\inf\{x\in[0,1]: \rho(t,x)>0\}, \quad  r_{\rho}(t):=\sup\{x\in[0,1]: \rho(t,x)<1\}.
\end{equation*}
Note that for  $t\ge t_{\rho_0}$ we have  $\ell_{\rho}(t)=  r_{\rho}(t)= 1-\alpha$ 
($\inf$ and $\sup$ have to be interpreted as essential extrema here, since $\rho(t,\cdot)$ is defined in $L^{\infty}$).
We have to prove the following generalization of Proposition \ref{th:fronteerlimit}.

\begin{proposition}\label{th:fronteerlimit2}
Let $p\in(1/2,1]$. For any $t\geq 0$, we have the following convergence in probability
 \begin{equation*}\begin{split}
  \lim_{N\to \infty} &N^{-1}L_{N,k}\left( \frac{Nt}{p-q}  \right)=\ell_{\rho}(t)\wedge \big(\ell + t\big),\\
  \lim_{N\to \infty} &N^{-1} R_{N,k}\left( \frac{Nt}{p-q}  \right)=r_{\rho}(t)\vee\big(r - t\big)\;.
\end{split}
  \end{equation*}
\end{proposition}
\noindent The asserted mixing time of Theorem \ref{th:asep2} is nothing but the first time $t\ge 0$ at which the limits obtained in this proposition reach the value $1-\alpha$.
This being given, the proof presented in Subsection \ref{Subsec:MixingAsep} works almost verbatim upon replacing $t_{\alpha,N}$ by
$$ t_N = \frac{N}{p-q} \max\big(t_\rho, 1-\alpha - \ell, r-1+\alpha\big)\;.$$
Indeed, the lower bound follows as a corollary of Proposition \ref{th:fronteerlimit2} and Lemma \ref{lem:localise}. Regarding the upper bound in the case $p<1$, we first notice that
$$ P_t^{\xi_N} - \pi_{N,k} = \sum_{\xi'\in\Omega_{N,k}}\pi_{N,k}(\xi') \big(P_t^{\xi_N} - P_t^{\xi'}\big)\;,$$
so that
$$ \| P_t^{\xi_N} - \pi_{N,k} \|_{TV}  \leq \pi_{N,k}\big(\Omega_{N,k}\backslash A^\gep_{N,k}\big) + \sup_{\xi'\in A^\gep_{N,k}} \| P_t^{\xi_N} - P_t^{\xi'} \|_{TV}\;.$$
The first term goes to $0$ by Lemma \ref{lem:localise}. Using the monotonicity under the grand coupling $\bbP$ at the second line, we bound the second term as follows
\begin{align*}
\sup_{\xi'\in A^\gep_{N,k}} \| P_t^{\xi_N} - P_t^{\xi'} \|_{TV} &\leq \sup_{\xi'\in A^\gep_{N,k}} \bbP\big[h_t^{\xi_N} \ne h^{\xi'}_t\big] \\
&\leq \bbP\big[h_t^{\xi_N} \ne h^{\vee}_t\big] + \bbP\big[h_t^{\xi_N} \ne h^{\wedge^\gep_{N,k}}_t\big]\;.
\end{align*}
Thus it suffices to show that both terms on the r.h.s. vanish when $t=t_N + \delta N$, for any $\delta > 0$ (notice that $t_N$ may be negligible compared to $N$). We write
\begin{align*}
\bbP\big[h_{t_N+\delta N}^{\xi_N} \ne h^{\vee}_{t_N+\delta N}\big] &\leq \bbP\big[h_{t_N+\delta N}^{\xi_N} \ne h^{\vee}_{t_N+\delta N}\ |\ h_{t_N}^{\xi_N} \in A^\gep_{N,k}\big]\\
&+ \bbP\big[h_{t_N}^{\xi_N} \notin A^\gep_{N,k}\big]\;,
\end{align*}
as well as
\begin{align*}
\bbP\big[h_{t_N+\delta N}^{\xi_N} \ne h^{\wedge^\gep_{N,k}}_{t_N+\delta N}\big] &\leq \bbP\big[h_{t_N+\delta N}^{\xi_N} \ne h^{\wedge^\gep_{N,k}}_{t_N+\delta N}\ |\ h_{t_N}^{\xi_N}, h^{\wedge^\gep_{N,k}}_{t_N} \in A^\gep_{N,k}\big]\\
&+ \bbP\big[h_{t_N}^{\xi_N} \notin A^\gep_{N,k}\big] + \bbP\big[h^{\wedge^\gep_{N,k}}_{t_N} \notin A^\gep_{N,k}\big]\;.
\end{align*}
From there, we can apply the same reasoning as in Subsection \ref{Subsec:MixingAsep} to show that these two terms go to $0$ as $N$ goes to $\infty$. This concludes the proof of Theorem \ref{th:asep2} when $p<1$. To treat the case $p=1$, one simply has to adapt the arguments presented in Subsection \ref{Subsec:p1}.
\qed

\section{Hydrodynamic limit}\label{Sec:Hydro}

In this section, we speed up the jump rates by a factor $N/(p-q)$ in order to simplify the notations. 
The previous notation $\eta^{\xi_N}(tN/(p-q),x)$ now becomes $\eta(t,x)$.

The theory of solutions of the Burgers equation with zero-flux boundary conditions was developed in B\"urger, Frid and Karlsen~\cite{BFK07}. As it is shown in~\cite{LabbeKPZ}, the unique solution of this PDE coincides with the unique entropy solution of the Burgers equation with appropriate Dirichlet boundary conditions:
\begin{align}\label{Eq:Burgers}
\begin{cases}
\partial_t \rho = -\partial_x \big( \rho(1-\rho) \big)\;,\quad t>0\;, x\in(0,1)\;,\\
\rho(t,0) = 0\;,\quad \rho(t,1) = 1\;,\\
\rho(0,\cdot) = \rho_0(\cdot)\;.
\end{cases}
\end{align}
Therefore, we only need to prove convergence towards the latter object.

\begin{remark}
As it is explained in~\cite{LabbeKPZ}, the particle system with zero-flux boundary conditions could essentially be obtained from the system on the whole line $\bbZ$ where we place only particles after site $N$ (density equal to $1$) and no particle before site $0$ (density equal to $0$): this provides a heuristic explanation for our Dirichlet boundary conditions above.
\end{remark}

The precise definition of the entropy solution of \eqref{Eq:Burgers} is the following. Hereafter $\langle f,g\rangle$ denotes the $L^2([0,1],dy)$ inner product of $f$ and $g$.

\begin{definition}\label{Def:Burgers}
Let $\rho_0 \in L^\infty([0,1],dy)$. We say that a function $(\rho(t,y),t\geq 0, y\in [0,1]) \in L^\infty(\bbR_+ \times [0,1],dt\otimes dy)$ is an entropy solution of the inviscid Burgers equation (\ref{Eq:Burgers}) if for all $\kappa \in [0,1]$ and all $\varphi\in \cC^\infty_c(\bbR_+\times\bbR,\bbR_+)$, we have
\begin{multline*}
\int_0^\infty \Big( \langle \partial_t \varphi(t,\cdot) , \big( \rho(t,\cdot) - \kappa \big)^\pm \rangle+  \langle \partial_y \varphi(t,\cdot) , h^\pm(\rho(t,\cdot),\kappa) \rangle\\
+ \varphi(t,0) (0-\kappa)^\pm + \varphi(t,1) (1-\kappa)^\pm \Big) dt
+ \langle \varphi(0,\cdot) , \big(\rho(0,\cdot) - \kappa \big)^\pm \rangle \geq 0\;,
\end{multline*}
where $h^+(r,\kappa) = \tun_{\bbR_+}(r-\kappa) \big(r(1-r) - \kappa(1-\kappa)\big)$ and $h^-(r,\kappa) = h^+(\kappa,r)$.
\end{definition}

For any given initial condition $\rho_0 \in L^\infty([0,1])$, there exists a unique entropy solution to the inviscid Burgers equation with Dirichlet boundary conditions, see 
Vovelle~\cite{Vovelle}. Let us mention that the original construction of entropy solutions is due to 
Bardos, Le Roux and N\'ed\'elec~\cite{Bardos} in the BV setting, and is extended to the $L^\infty$ 
setting by Otto~\cite{Otto}. One should notice that the solution does 
not necessarily satisfy the boundary conditions but instead, satisfies the so-called BLN conditions. In particular, 
we do not expect the solution to be equal to $1$ at $x=1$ for short times when it starts from $\wedge$: indeed, 
it takes a macroscopic time for the rightmost particle to reach the right boundary and therefore the density of particles at $x=1$ remains null for a while.

Let $(\xi_N)_{N\geq 1}$ be a sequence of initial conditions and let $h(\xi_N)$ be the associated sequence of height functions. Observe that $(u_0^N(x):=\frac1{N} h(\xi_N)(xN), x\in [0,1])$ is $1$-Lipschitz so that it is tight in $\cC([0,1])$.

\begin{lemma}\label{Lemma:Tightness}
For any choice of initial conditions, the sequence of processes $\rho^N$ and $u^N$ are $\bbC$-tight in $\bbD([0,\infty),\cM)$ and $\bbD([0,\infty),\cC([0,1]))$ respectively.
\end{lemma}
\begin{proof}
The arguments are standard so we only give a sketch of proof. To prove tightness of $\rho^N$ it suffices to show that
\begin{equation*}
\varlimsup_{\delta \downarrow 0} \varlimsup_{N\rightarrow \infty} \E\Big[ \sup_{s,t\leq T, |t-s|\leq \delta} \big|\langle \rho^N_t-\rho^N_s,\varphi\rangle\big|\Big] = 0\;,
\end{equation*}
for all $\varphi\in \cC^\infty([0,1])$. To that end, we write
\begin{equation}\label{Eq:Dynkin}
\langle \rho^N_t,\varphi\rangle - \langle \rho^N_s,\varphi\rangle = \int_s^t L^N \langle \rho^N_r,\varphi\rangle dr + M^N_{s,t}\;,
\end{equation}
where $M^N_{s,\cdot}$ is a martingale and $L^N$ is the sped-up generator of our process. Then, it is simple to bound the two terms on the r.h.s, see for instance~\cite[Lemma 4.1]{Reza}.\\
Regarding the tightness of the sequence $u^N$, we first observe that for all $t\geq 0$, the profiles $y\mapsto u^N(t,y)$ are $1$-Lipschitz. Furthermore, for every $y\in [0,1]$, the process $t\mapsto u^N(t,y)$ makes jumps of size at most $2/N$ and at rate at most $N/(p-q)$. Using the moments formula for the Poisson r.v., we deduce that for all $m\geq 1$ we have
$$ \sup_{y\in[0,1]} \E\big[ \big|u^N(t,y)-u^N(s,y)\big|^m\big]^\frac{1}{m} \lesssim \frac{|t-s|^{\frac1{m}}}{N^{\frac{m-1}{m}}} + |t-s|\;,$$
uniformly over all $s,t\geq 0$ and all $N\geq 1$. Then, one defines $\bar{u}^N$ as the continuous-time interpolation of $u^N$ taken at all times $t \in \bbZ/N$. It is straightforward to check that $\bar{u}^N$ is tight using the estimate already obtained on $u^N$. To conclude, we only need to check that $u^N$ and $\bar{u}^N$ are uniformly close on any given compact space-time set: this can be done by bounding the moments of the supremum of $|u^N-\bar{u}^N|$ on boxes of size $1/N\times 1/N$ and then summing over a covering of the given compact set into such boxes, see for instance~\cite[Lemma 4.7]{BG97}.
\end{proof}


Below, we will consider an initial probability measure $\iota_N$ on $\cup_k \Omega_{N,k}^0$ that satisfies the following assumption.

\begin{assumption}\label{Assumption:Density}
There exists a piecewise constant function $f:[0,1]\rightarrow [0,1]$ such that for all $N\geq 1$, $\iota_N = \otimes_{x=1}^N \Be(f(x/N))$ where $\Be(c)$ denotes the Bernoulli $\pm 1$ distribution with parameter $c$.
\end{assumption}

The main step of the proof of Theorem \ref{Th:HydroGene} consists in establishing the hydrodynamic limit starting from elementary initial conditions.

\begin{theorem}\label{Th:HydroDensity}
We work under Assumption \ref{Assumption:Density} and we let $f$ be the density appearing therein. The sequence of empirical densities $\rho^N$ converges in probability in the Skorohod space $\bbD([0,\infty),\cM)$ to the deterministic process $\rho(t,dy) = \rho(t,y)dy$ where $(\rho(t,y),y\in[0,1],t\geq 0)$ is the entropy solution of (\ref{Eq:Burgers}) with initial condition $\rho(0,\cdot) = f(\cdot)$.
\end{theorem}

Given this result, we turn to the proof of Theorem \ref{Th:HydroGene}.

\begin{proof}[Proof of Theorem \ref{Th:HydroGene}]
Since we assume that $\rho_0^N$ converges weakly to $\rho_0$, it is not difficult to deduce that
\begin{equation}\label{Eq:ICHeight}
\lim_{N\rightarrow\infty} \sup_{y\in [0,1]} \big| u_0^N(y) - u_0(y) \big| = 0\;,
\end{equation}
where $u_0(y) = \int_0^y (2\rho_0(x)-1)dx$. Thanks to Lemma \ref{Lemma:Tightness}, we only need to check that any limiting points $\rho$ of $\rho^N$ is the entropy solution of the Burgers equation (\ref{Eq:Burgers}) starting from $\rho_0$ and that any limiting point $u$ of $u^N$ is its integrated version. The latter property is actually simple to establish once we know that $\rho^N$ converges to $\rho$ so we concentrate on this convergence. 

In the case where the sequence of initial conditions $\rho^N_0$ satisfies Assumption \ref{Assumption:Density}, Theorem \ref{Th:HydroDensity} yields the convergence of $\rho^N$ towards the entropy solution of (\ref{Eq:Burgers}). To extend the scope of this convergence result to a general sequence of initial conditions $\xi_N$, we proceed by approximation. Set $\alpha=\big(1+u_0(1)\big)/2$, and recall that $u_0$ is $1$ Lipschitz so that $\rho_0(x)$ is almost everywhere in $[0,1]$. Fix $\epsilon>0$. One can find two profiles $u^{+,\epsilon}_0, u^{-,\epsilon}_0$ which are $1$-Lipschitz, start from $0$ at $0$, are piecewise affine and 
satisfy the inequalities (see Figure \ref{fig:lesepsilons}):

\begin{figure}[htbp]
\centering
\leavevmode
\begin{center}
 \resizebox{7cm}{!}{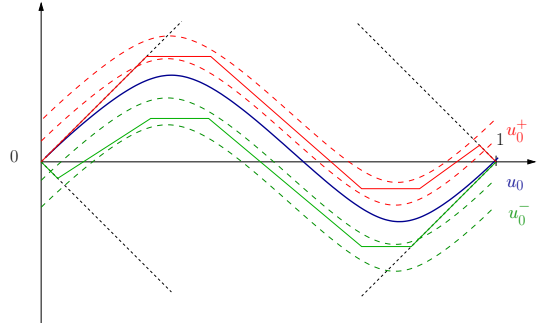}
 \end{center}
\caption{ How to bound $u_0$ by two piecewise affine functions: an important feature of the construction is that we force the 
slope of $u^{\pm}_0$ to be $\pm 1$ near the boundary so that the inequality still holds with high probability when the corresponding microscopic height functions are compared.
}
\label{fig:lesepsilons}
\end{figure}
\begin{equation}\label{Eq:BoundsApprox}\begin{split}
u_0(y)-2\epsilon &\leq u^{-,\epsilon}_0(y) \leq \big(u_0(y)-\epsilon\big)\vee(-y)\vee(y-2+2\alpha)\;,\\
\big(u_0(y)+\epsilon\big)\wedge y \wedge (2\alpha-y) &\leq u^{+,\epsilon}_0(y) \leq u_0(y)+2\epsilon\;,
\end{split}\end{equation}
and are such that $\| \rho^{\pm,\epsilon}_0 - \rho_0\|_{L^1}$ goes to $0$ as $\epsilon \downarrow 0$, where $\rho^{\pm,\epsilon}_0 := (\partial_y u^{\pm,\epsilon}_0+1)/2$. Then, for every $N\geq 1$ we consider three initial configurations of the ASEP: one is given by $\xi_N$, the two others $\xi_N^{\pm,\epsilon}$ are random elements in $\cup_k \Omega_{N,k}^0$ with law $\otimes_{x=1}^N \mbox{Be}(\rho_0^{\pm,\epsilon}(x/N))$. We couple these three ASEP in such a way that the order on the height functions is preserved by the dynamics (similarly as in our grand coupling). Equation \eqref{Eq:BoundsApprox} ensures that the probability of the event
$$ h(\xi_N^{-,\epsilon})(t=0,\cdot) \leq h(\xi_N)(t=0,\cdot) \leq h(\xi_N^{+,\epsilon})(t=0,\cdot) \;,$$
goes to $1$ as $N\rightarrow\infty$. Therefore, the probability that these inequalities happen at all times $t\ge 0$ goes to $1$ as well.

Our convergence result applies to $\rho_N^{\pm,\epsilon}$ (which are the empirical densities associated to $h(\xi_N^{\pm,\epsilon})$): the limits $\rho^{\pm,\epsilon}$ are the solutions of (\ref{Eq:Burgers}) starting from $\rho_0^{\pm,\epsilon}$. Let us denote by $u^{\pm,\epsilon}$ the associated integrated solutions. Our coupling ensures that any limit point $u$ of the tight sequence $u^N$ lies in between these two integrated solutions. By the $L^1$ contractivity~\cite[Th. 7.28]{Malek} of the solution to (\ref{Eq:Burgers}), one deduces that as $\epsilon \downarrow 0$, $u^{\pm,\epsilon}$ converges to the integrated solution of (\ref{Eq:Burgers}) starting from $u_0$, thus $u$ coincides with this solution and this concludes the proof.
\end{proof}

We are now left with proving Theorem \ref{Th:HydroDensity}. To that end, it suffices to show that any limit point of the tight sequence $\rho^N$ satisfies the entropy inequalities of Definition \ref{Def:Burgers}. The usual trick that makes the contant $\kappa$ appear in these inequalities is to couple $\eta$ with another particle system $\zeta$ which is stationary with distribution $\otimes_{x=1}^N \mbox{Be}(\kappa)$. Actually, our boundary conditions complicate the proof and it will be convenient to consider another particle system which evolves according to the same dynamics but on the whole line $\bbZ$.

More precisely, we will consider the process $(\eta,\zeta,\heta,\hzeta)$ where each element of the quadruplet is a process that lives in $\{0,1\}^\bbZ$ and such that the following holds. The restriction of $\eta$ to $\lint 1, N\rint$ is a Markov process evolving according to the ASEP dynamics considered from the beginning of this paper, while the restriction of $\eta$ to $\bbZ\backslash\lint 1, N\rint$ remains constant. The process $\zeta$ remains constant outside $\lint 1, N\rint$, while in $\lint 1, N\rint$ it undergoes the same dynamics as $\eta$ except that at site $1$ if there is no particle, then a particle is created at rate $N(2p-1)\kappa/(p-q)$ and at site $N$, if there is a particle, then it is removed at rate $N(2p-1)(1-\kappa)/(p-q)$. Regarding $\heta$ and $\hzeta$, they evolve according to the ASEP $(p,q)$ on the whole line $\bbZ$ without any boundary effect: the dynamics is translation invariant.

Let us now explain how the processes are coupled. For all pairs of consecutive sites in $\lint 1, N\rint$, we make the jumps simultaneous for the four particle systems. For all pairs of consecutive sites in $
\bbZ\backslash\lint 2,N-1\rint$, we make the jumps simultaneous for $\heta$ and $\hzeta$. Let us point out that each of these processes is Markov. Instead of writing down the generator $\tilde{\cL}$ of the quadruplet acting on a general test function, we restrict to test functions involving only two of the four processes. We set $b(x,y) := x(1-y)$. The generator acting on $\eta,\zeta$ is given by
$$ \tilde{\cL}f(\eta,\zeta) = \tilde{\cL}^{\mbox{\tiny bulk}}f(\eta,\zeta) + \tilde{\cL}^{\mbox{\tiny bdry}} f(\eta,\zeta)\;,$$
where
\begin{align*}
\tilde{\cL}^{\mbox{\tiny bulk}}f(\eta,\zeta) &= \frac{N}{p-q}\sum_{k,\ell =1}^{N} \big(p\tun_{\{\ell-k=1\}} +q\tun_{\{k-\ell=1\}}\big) \cG_{k,\ell}(\eta,\zeta)\;,
\end{align*}
where
\begin{align*}
 \cG_{k,\ell}(\eta,\zeta) &:= \big( b(\eta(k),\eta(\ell))\wedge b(\zeta(k),\zeta(\ell)) \big)\big( f(\eta^{k,\ell},\zeta^{k,\ell}) - f(\eta,\zeta)\big)\\
& + \big(b(\eta(k),\eta(\ell))- b(\eta(k),\eta(\ell))\wedge b(\zeta(k),\zeta(\ell)) \big)\big( f(\eta^{k,\ell},\zeta) - f(\eta,\zeta)\big)\\
& + \big( b(\zeta(k),\zeta(\ell))- b(\eta(k),\eta(\ell))\wedge b(\zeta(k),\zeta(\ell)) \big)\big( f(\eta,\zeta^{k,\ell}) - f(\eta,\zeta)\big)\;,
\end{align*}
and, using the notation $\zeta \pm \delta_k$ to denote the particle configuration which coincides with $\zeta$ everywhere except at site $k$ where the occupation is taken to be $\zeta(k) \pm 1$,
\begin{align*}
\tilde{\cL}^{\mbox{\tiny bdry}} f(\eta,\zeta) &= \frac{N}{p-q}(2p-1) \kappa (1-\zeta(1)) \big( f(\eta,\zeta+\delta_1) - f(\eta,\zeta)\big)\\
& + \frac{N}{p-q}(2p-1) (1-\kappa) \zeta(N) \big( f(\eta,\zeta-\delta_{N}) - f(\eta,\zeta)\big)\;.
\end{align*}
The generator acting on $\heta,\hzeta$ is given by
\begin{align*}
\tilde{\cL}f(\heta,\hzeta) =  \frac{N}{p-q}\sum_{k,\ell \in \bbZ} \big(p\tun_{\{\ell-k=1\}} +q\tun_{\{k-\ell=1\}}\big) \cG_{k,\ell}(\heta,\hzeta)\;.
\end{align*}
Let us also provide the expression of the generator acting on $\eta,\heta$:
\begin{align*}
\tilde{\cL}f(\eta,\heta) &=  \frac{N}{p-q}\sum_{k,\ell =1}^N \big(p\tun_{\{\ell-k=1\}} +q\tun_{\{k-\ell=1\}}\big) \cG_{k,\ell}(\eta,\heta)\\
& +\frac{N}{p-q}\!\!\sum_{k,\ell \in \bbZ\backslash \rint 2,,N-1\rint} \!\!\!\big(p\tun_{\{\ell-k=1\}} +q\tun_{\{k-\ell=1\}}\big)\big(f(\eta,\heta^{k,\ell})-f(\eta,\heta)\big)\;.
\end{align*}

Let us introduce an initial condition that will be useful in the sequel. Let $\iota_N$ be a measure on $\cup_k \Omega_{N,k}^0$ that satisfies Assumption \ref{Assumption:Density}. Let the restriction of $\eta_0$ to 
$\lint 1, N\rint$ start with law $\iota_N$ and let the restriction of $\zeta_0$ to $\lint 1, N\rint$ start with law given by a product of Bernoulli measures $\mbox{Be}(\kappa)$. We also let $\eta_0(x)=\zeta_0(x)=0$ for all $x\in \bbZ\backslash \lint 1, N\rint$. These two initial conditions are coupled in the following way: for every $x\in\lint 1, N\rint$, we have $\eta_0(x) \geq \zeta_0(x)$ if and only if $f(x/N) \geq \kappa$, where $f$ is the density arising in Assumption \ref{Assumption:Density}. Additionally, we set $\heta_0(x) = \eta_0(x)$ for all $x\in\bbZ$. Finally, we set $\hzeta_0(x) = \zeta_0(x)$ for all $x\in\lint 1, N\rint$, and we draw the remaining values of $\hzeta_0(\cdot)$ according to an independent sequence of $\mbox{Be}(\kappa)$. The law of the quadruplet starting from this initial condition will be denoted by $\bbQ^N_{\iota_N,\kappa}$.\\

The first step consists in establishing the entropy inequalities at the microscopic level. Recall the notation $b(x,y)=x(1-y)$. We define a function $F_{k,\ell}$ acting on a pair of particle systems on $\bbZ$ as follows. We set $F_{k,\ell}(\eta,\zeta)=1$ if $\eta(k)\geq \zeta(k)$ and $\eta(\ell)\geq \zeta(\ell)$; otherwise we set $F_{k,\ell}(\eta,\zeta)=0$. Moreover we set
$$ H^+(\eta,\zeta) = \big(b(\eta(1),\eta(0)) - b(\zeta(1),\zeta(0))\big) F_{1,0}(\eta,\zeta)\;,\quad H^-(\eta,\zeta) = H^+(\zeta,\eta)\;. $$
We also let
$$ \langle f,g \rangle_N := \frac1{N}\sum_{k\in\bbZ} f(k)g(k)\;.$$

\begin{lemma}[Microscopic inequalities]\label{Lemma:MicroIneq}
We work under Assumption \ref{Assumption:Density}. For all $\varphi\in\cC^\infty_c(\bbR_+\times\bbR,\bbR_+)$, all $\delta > 0$ and all $\kappa\in[0,1]$, we have $\lim_{N\rightarrow\infty} \bbQ^N_{\iota_N,\kappa}\big(I_{\mbox{\tiny micro}} \geq - \delta\big) = 1$ where $I_{\mbox{\tiny micro}}$ denotes either
\begin{align*}
&\int_0^\infty \Big(\Big\langle \partial_t \varphi(t,\cdot) , \big(\eta(t,\cdot) - \zeta(t,\cdot)\big)^\pm \Big\rangle_N + \Big\langle \partial_x \varphi(t,\cdot) ,
H^\pm(\tau_\cdot \eta(t), \tau_\cdot \zeta(t)) \Big\rangle_N \\
&+ \big((0-\kappa)^\pm \varphi(t,0) + (1-\kappa)^\pm \varphi(t,1)\big)\Big)dt+ \Big\langle \varphi(0,\cdot) , \big(\eta(0,\cdot) - \zeta(0,\cdot)\big)^\pm \Big\rangle_N\;,
\end{align*}
or
\begin{align*}
\int_0^\infty &\Big(\Big\langle \partial_t \varphi(t,\cdot) , \big(\heta(t,\cdot) - \hzeta(t,\cdot)\big)^\pm \Big\rangle_N+ \Big\langle \partial_x \varphi(t,\cdot) , H^\pm(\tau_\cdot \heta(t), \tau_\cdot \hzeta(t)) \Big\rangle_N\Big)dt \\
&\quad+ \Big\langle \varphi(0,\cdot) , \big(\heta(0,\cdot) - \hzeta(0,\cdot)\big)^\pm \Big\rangle_N\;.
\end{align*}
\end{lemma}
\begin{proof}
This is similar to Lemma 2.10 in~\cite{LabbeKPZ}. Let us recall the main steps here in the case of $(\eta,\zeta)$: the case of $(\heta,\hzeta)$ is simpler since we don't have to deal with the boundary terms. First of all, we set
\begin{multline*}
B_t \! = \!\! \int_0^t \! \Big(\Big\langle \partial_s \varphi(s,\cdot) , \big(\eta(s,\cdot) - \zeta(s,\cdot)\big)^\pm \Big\rangle_N  \!
+ \tilde{\cL}\Big\langle \varphi(s,\cdot) , \big(\eta(s,\cdot) - \zeta(s,\cdot)\big)^\pm \Big\rangle_N \Big)ds \\
+ \Big\langle \varphi(0,\cdot) , \big(\eta(0,\cdot) - \zeta(0,\cdot)\big)^\pm \Big\rangle_N\;.
\end{multline*}
By definition of the generator, we have the identity
$$ \Big\langle \varphi(t,\cdot) , \big(\eta(t,\cdot) - \zeta(t,\cdot)\big)^\pm \Big\rangle_N = B_t + M_t\;,$$
where $M$ is a mean-zero martingale. A long calculation shows that the term in $B_t$ involving the generator is bounded by
$$\Big\langle \partial_x \varphi(s,\cdot) , H^\pm(\tau_\cdot \eta(s), \tau_\cdot \zeta(s)) \Big\rangle_N+ (0-\kappa)^\pm \varphi(s,0) + (1-\kappa)^\pm \varphi(s,1) \;,$$
up to a negligible term of order $1/N$. Moreover, the Burkholder-Davis-Gundy inequality allows one to bound the moments of the martingale and to show that they vanish as $N\rightarrow\infty$. Using the fact that $\varphi$ is compactly supported, one gets $B_t = -M_t$ for $t$ large enough. The assertion of the lemma follows by putting everything together.
\end{proof}

The next step consists in replacing the microscopic quantities by averages on boxes of size $\ell$. We denote by $T_\ell(k) := \lint k-\ell,k+\ell\rint$ and we set 
$$ \ccM_{T_\ell(k)} f = \frac{1}{2\ell+1} \sum_{i\in T_\ell(k)} f(i)\;,$$
for any map $f:\bbZ\rightarrow\bbR$. The invariance by translation of the dynamics of $(\heta,\hzeta)$ ensures that for any $A,T>0$ and for \textit{any} initial condition $(\heta_0,\hzeta_0)$ such that $\hzeta_0$ is a product of $\mbox{Be}(\kappa)$, one has
\begin{equation}\label{Eq:MicroRepl}\begin{split}
\lim_{\ell\rightarrow\infty}\lim_{N\rightarrow\infty} &\bbQ^N \bigg[\int_0^T \frac1N \sum_{k = -A N}^{A N} \Big|\ccM_{T_\ell(k)} (\heta(t)- \hzeta(t))^\pm\\
&\qquad- (\ccM_{T_\ell(k)} \heta(t)- \kappa)^\pm\Big| dt \bigg] = 0\;,\\
\lim_{\ell\rightarrow\infty}\lim_{N\rightarrow\infty} &\bbQ^N \bigg[\int_0^T \frac1N \sum_{k = -A N}^{A N} \Big|\ccM_{T_\ell(k)} H^\pm(\heta(t), \hzeta(t))\\
&\qquad- h^\pm(\ccM_{T_\ell(k)} \heta(t), \kappa)\Big| dt \bigg] = 0\;,
\end{split}\end{equation}
see the arguments on pp.426-427 of Rezakhanlou~\cite{Reza}. These arguments do not apply anymore to $(\eta,\zeta)$. However the next lemma shows that $\eta-\heta$ and $\zeta-\hzeta$ are small in the bulk of the lattice, and therefore, one deduces that (\ref{Eq:MicroRepl}) also holds with $(\heta,\hzeta)$ replaced by $(\eta,\zeta)$. In the lemma below, we let $\bbQ$ be the law of the dynamics starting from some deterministic initial condition $(\eta_0,\zeta_0,\heta_0,\hzeta_0)$.

\begin{lemma}\label{Lemma:Coupling}
There exits a constant $C>0$ such that for all $\epsilon>0$, we have
$$ \lim_{N\rightarrow \infty}\sup_{\eta_0=\heta_0\in\{0,1\}^\bbZ} \sup_{s\in [0,C\epsilon]} \bbQ\Big[ \frac1{N} \sum_{x\in [\epsilon N, N-\epsilon N]} \big|\eta(s,x) - \heta(s,x)\big| \Big] = 0\;, $$
and similarly with $(\zeta,\hzeta)$.
\end{lemma}
This lemma is in the spirit of~\cite[Lemma 3.3]{Baha}.
\begin{proof}
Fix $\epsilon > 0$. Let $\varphi$ be a function from $\bbR_+\times\bbR$ into $\bbR_+$ such that:\begin{enumerate}
\item $\varphi(t,x)=0$ as soon as $x\notin[\epsilon/3,1-\epsilon/3]$,
\item $\varphi(t,x) > c$ for all $(t,x) \in [0,C\epsilon]\times [\epsilon,1-\epsilon]$ and for some $c,C>0$,
\item $\partial_t \varphi + 2\big|\partial_x \varphi\big| \leq 0$.
\end{enumerate}
Such a function exists. Take for instance $\varphi(t,x) = \Phi(6t \epsilon^{-1} + q(x))$ where $\Phi:\bbR\rightarrow [0,1]$ is smooth, non-increasing, equal to $1$ on $\bbR_-$ and to $0$ on $[1,\infty)$, and $q:\bbR\mapsto [0,1]$ is equal to $0$ on $[\epsilon,1-\epsilon]$, to $1$ on $(-\infty,\epsilon/3]\cup[1-\epsilon/3,\infty)$ and is such that $\|q'\|_\infty \leq 3\epsilon^{-1}$.\\
Since $\eta_0=\hat{\eta}_0$ and $\varphi$ vanishes on the boundaries, we have
\begin{align}\label{Eq:LemmaCoupling}
\big\langle \varphi(t,\cdot),(\eta_t - \hat{\eta}_t)^\pm\big\rangle_N = I_t^\pm + M_t^\pm\;,
\end{align}
where $M_t^\pm$ is a martingale and
$$I_t^\pm=\int_0^t \Big( \big\langle \partial_s\varphi(s,\cdot),(\eta_s - \hat{\eta}_s)^\pm\big\rangle_N +  \tilde{\cL}\big\langle \varphi(s,\cdot),(\eta_s - \hat{\eta}_s)^\pm\big\rangle_N\Big) ds\;.$$
Using the computation in the proof of Lemma \ref{Lemma:MicroIneq}, we obtain
\begin{align*}
I_t^\pm &\leq \int_0^t \Big( \big\langle \partial_s \varphi(s,\cdot) , 
(\eta_s - \hat{\eta}_s)^\pm \big\rangle_N + \big\langle \partial_x\varphi(s,\cdot) , 
H^\pm(\tau_\cdot \eta_s, \tau_\cdot \hat{\eta}_s) \big\rangle_N \Big) ds\\
&+ C_{\phi} N^{-1}\;,
\end{align*}
where $C_{\phi}$ is a positive constant which depends only $\phi$ and could be made explicit (we will use the same notation for similar constant, as we do not believe it should yield confusion).

Let $H:=H^+ + H^-$. By considering all possible configurations, we can check that
$$ \big| H(\tau_k \eta_s , \tau_k \hat{\eta}_s) \big| \leq | \eta_s(k)-\hat{\eta}_s(k)| + | \eta_s(k+1)-\hat{\eta}_s(k+1)|\;. $$
Furthermore, we have
\begin{align*}
\big\langle \partial_s \varphi(s,\cdot) , |\eta_s - \hat{\eta}_s| \big\rangle_N &= \big\langle \partial_s \varphi(s,\cdot) , \frac{| \eta_s(\cdot)-\hat{\eta}_s(\cdot)| + | \eta_s(\cdot+1)-\hat{\eta}_s(\cdot+1)|}{2}\big\rangle_N\\
&+  C_{\phi} N^{-1}\;,
\end{align*}
uniformly over all $s\geq 0$ and all $N\geq 1$. Consequently, we get
\begin{align*}
I_t^+ + I_t^- &\leq  C_{\phi} N^{-1}+ \int_0^t \Big\langle \partial_s \varphi(s,\cdot)\\
&\quad+ 2 \big|\partial_x \varphi(s,\cdot)\big| , \frac{| \eta_s(\cdot)-\hat{\eta}_s(\cdot)| + | \eta_s(\cdot+1)-\hat{\eta}_s(\cdot+1)|}{2} \Big\rangle_N ds\\
&\leq  C_{\phi} N^{-1}\;,
\end{align*}
uniformly over all $t$ in a compact set. Using the properties of the function $\varphi$, we deduce that uniformly over all $t\in [0,C\epsilon]$
\begin{align*}
&\bbQ\Big[ \frac1{N} \sum_{x\in [\epsilon N, N-\epsilon N]} \big|\eta(t,x) - \heta(t,x)\big| \Big]\\
&\leq \frac1{c}\bbQ\Big[ \big\langle \varphi(t,\cdot),|\eta_t - \hat{\eta}_t|\big\rangle_N\Big]\leq \frac1{c}\bbQ[M_t^+ + M_t^-] + C_{\phi} N^{-1}\le  C_{\phi} N^{-1} \;,
\end{align*}
so that the statement of the lemma follows.
\end{proof}

\begin{lemma}\label{Lemma:IneqMeso}
We work under Assumption \ref{Assumption:Density}. For all $\varphi\in\cC^\infty_c(\bbR_+\times\bbR,\bbR_+)$, all $\delta > 0$ and all $\kappa\in[0,1]$, we have $\lim_{\ell\rightarrow\infty}\lim_{N\rightarrow\infty} \bbQ^N_{\iota_N,\kappa}(I_{\mbox{\tiny meso}} \geq -\delta) = 1$ where $I_{\mbox{\tiny meso}}$ is given by
\begin{align*}
&\int_0^\infty \Big(\Big\langle \partial_t \varphi(t,\cdot) , \big(\ccM_{T_\ell(\cdot)}\eta(t) - \kappa\big)^\pm \Big\rangle_N+ \Big\langle \partial_x \varphi(t,\cdot) , h^\pm(\ccM_{T_\ell(\cdot)} \eta(t), \kappa) \Big\rangle_N \\
&+ \big((0-\kappa)^\pm \varphi(t,0) + (1-\kappa)^\pm \varphi(t,1)\big)\Big)dt+ \Big\langle \varphi(0,\cdot) , \big(\ccM_{T_\ell(\cdot)}\eta(0) - \kappa\big)^\pm \Big\rangle\;.
\end{align*}
\end{lemma}
\begin{proof}
From Lemma \ref{Lemma:MicroIneq} and the smoothness of $\varphi$, we deduce that
$$ \lim_{\ell\rightarrow\infty}\lim_{N\rightarrow\infty} \bbQ^N_{\iota_N,\kappa}(J_{\mbox{\tiny micro}} \geq -\delta) = 1\;,$$
where $J_{\mbox{\tiny micro}}$ is given by
\begin{align*}
\int_0^\infty &\Big(\frac1{N} \sum_{k=1}^N \Big( \partial_t \varphi(t,k) \ccM_{T_\ell(k)}\big(\eta(t) - \zeta(t)\big)^\pm + \partial_x \varphi(t,k) \ccM_{T_\ell(k)} H^\pm(\eta(t),\zeta(t)) \Big) \\
&+ \big((0-\kappa)^\pm \varphi(t,0) + (1-\kappa)^\pm \varphi(t,1)\big)\Big)dt
&+\frac1{N} \sum_{k=1}^N \varphi(0,k) \ccM_{T_\ell(k)}\big(\eta(0) - \zeta(0)\big)^\pm \;.
\end{align*}
Since $\varphi$ is compactly supported, we can restrict the time integral to $[0,T]$ for some large enough $T>0$. The statement of the lemma follows if we are able to show that as $N\rightarrow\infty$ and $\ell\rightarrow\infty$ the $\bbQ^N_{\iota_N,\kappa}$-expectations of the following three quantities vanish
$$\int_0^T \frac1{N} \sum_{k=1}^N \Big| \ccM_{T_\ell(k)}\big(\eta(t) - \zeta(t)\big)^\pm - \big(\ccM_{T_\ell(k)}\eta(t) - \kappa\big)^\pm\Big| dt\;,$$
$$\int_0^T \frac1{N} \sum_{k=1}^N \Big| \ccM_{T_\ell(k)}H^\pm(\eta(t),\zeta(t)) - h^\pm(\ccM_{T_\ell(k)} \eta(t), \kappa)\Big| dt\;,$$
$$\frac1{N} \sum_{k=1}^N \Big| \ccM_{T_\ell(k)}\big(\eta(0) - \zeta(0)\big)^\pm - \big(\ccM_{T_\ell(k)}\eta(0) - \kappa\big)^\pm\Big|\;.$$
For the third one, it suffices to use the fact that the number of sign changes of $k\mapsto \eta(0,k)-\zeta(0,k)$ is uniformly bounded over $N\geq 1$ (this is a consequence of our coupling of the initial conditions). We now concentrate on the convergence of the second expression, since the convergence of the first follows from similar arguments. Fix $\epsilon > 0$. Let $(\eta,\zeta,\heta,\hzeta)$ be the process defined previously in this section except that at every time $t_i = i C\epsilon$, $i\geq 1$, we reinitialise $\heta$ and $\hzeta$ by letting them be equal to $\eta$ and $\zeta$ at this same time. We let $\bbQ^N$ be the law of the corresponding process. Then we write,
\begin{align}\label{Eq:MesoDecompo}\begin{split}
&\Big| \ccM_{T_\ell(k)}H^\pm(\eta(t),\zeta(t)) - h^\pm(\ccM_{T_\ell(k)} \eta(t), \kappa)\Big|\\
&\leq \Big| \ccM_{T_\ell(k)}H^\pm(\eta(t),\zeta(t)) - \ccM_{T_\ell(k)}H^\pm(\heta(t),\zeta(t))\Big|\\
&\quad+ \Big| \ccM_{T_\ell(k)}H^\pm(\heta(t),\zeta(t)) - \ccM_{T_\ell(k)}H^\pm(\heta(t),\hzeta(t))\Big|\\
&\quad+ \Big|\ccM_{T_\ell(k)}H^\pm(\heta(t),\hzeta(t)) - h^\pm(\ccM_{T_\ell(k)} \heta(t), \kappa)\Big|\\
&\quad+ \Big|h^\pm(\ccM_{T_\ell(k)} \heta(t), \kappa)-h^\pm(\ccM_{T_\ell(k)} \eta(t), \kappa)\Big|\;,\end{split}
\end{align}
and we bound separately the contributions coming from the terms arising on the right hand side. Notice that (\ref{Eq:MicroRepl}) still holds for $(\heta,\hzeta)$ as long as we apply it to interval of times of the form $[t_i,t_{i+1})$ (since our modified dynamics coincides with the original one on these intervals). Therefore, for every $i\geq 0$
\begin{align*}
\lim_{\ell\rightarrow\infty}\lim_{N\rightarrow\infty} \bbQ^N\!\bigg[ \int_{t_i}^{t_{i+1}}\!\! \frac1{N} \sum_{k=1}^N \Big|\ccM_{T_\ell(k)}H^\pm(\heta(t),\hzeta(t)) - h^\pm(\ccM_{T_\ell(k)} \heta(t), \kappa)\Big| dt \bigg]= 0\;,
\end{align*}
and we deduce that
\begin{align*}
\lim_{\ell\rightarrow\infty}\lim_{N\rightarrow\infty} \bbQ^N\bigg[ \int_0^T \frac1{N} \sum_{k=1}^N \Big|\ccM_{T_\ell(k)}H^\pm(\heta(t),\hzeta(t)) - h^\pm(\ccM_{T_\ell(k)} \heta(t), \kappa)\Big| dt \bigg] = 0\;.
\end{align*}
On the other hand, as long as $\ell \leq \epsilon N$ we write
\begin{align*}
&\int_0^T \frac1{N} \sum_{k=1}^N \Big| \ccM_{T_\ell(k)}H^\pm(\eta(t),\zeta(t)) - \ccM_{T_\ell(k)}H^\pm(\heta(t),\zeta(t))\Big| dt\\
&\lesssim 4\epsilon T + \sum_{i=0}^{\lfloor T/C\epsilon \rfloor} \int_{t_i}^{t_{i+1}} \frac1{N} \sum_{k=\lceil \epsilon N \rceil}^{N-\lfloor\epsilon N\rfloor} \Big| H^\pm(\tau_k\eta(t),\tau_k\zeta(t)) - H^\pm(\tau_k\heta(t),\tau_k\zeta(t))\Big| dt\;.
\end{align*}
To bound the second term on the right hand side, we first notice that there exists $K>0$ such that for any particle configurations $\eta_1,\eta_2,\eta'_1$ and $\eta'_2$, we have
$$ |H^\pm(\eta_1,\eta_2) - H^\pm(\eta'_1,\eta'_2)| \leq K \sum_{j=0,1} \sum_{m=1,2} |\eta_m(j)-\eta'_m(j)|\;. $$
Consequently, Lemma \ref{Lemma:Coupling} ensures that
\begin{align*}
\lim_{\ell\rightarrow\infty}\lim_{N\rightarrow\infty} \sum_{i=0}^{\lfloor T/\epsilon \rfloor} \bbQ^N \bigg[\int_{t_i}^{t_{i+1}} \frac1{N} &\sum_{k=\lceil \epsilon N \rceil}^{N-\lfloor\epsilon N\rfloor} \Big| H^\pm(\tau_k\eta(t),\tau_k\zeta(t))\\
&- H^\pm(\tau_k\heta(t),\tau_k\zeta(t))\Big| dt\bigg] = 0\;.
\end{align*}
The same argument allows to control the second term in (\ref{Eq:MesoDecompo}). Regarding the fourth term, it suffices to use the Lipschitz continuity of $h^\pm$ and Lemma \ref{Lemma:Coupling}. This concludes the proof.
\end{proof}
\begin{proof}[Proof of Theorem \ref{Th:HydroDensity}]
The two-blocks estimate~\cite[Lemma 6.6]{Reza} ensures that one can replace averages on boxes of size $\ell$ by averages on boxes of size $\epsilon N$. Therefore, we deduce that the conclusion of Lemma \ref{Lemma:IneqMeso} still holds upon such a replacement. Finally, one relies on classical arguments to show that this is sufficient to get the entropy inequalities, we refer the interested reader to~\cite[Proof of Th 2.7]{LabbeKPZ} for the details.
\end{proof}

\section{Locating the leftmost particle using the hydrodynamic profile}\label{sec:lmprmes}

In this section we explain how Propositions \ref{th:fronteerlimit} and \ref{th:fronteerlimit2} can be deduced from the hydrodynamic limit of the height function.
The case $\alpha=0$ is a bit particular as, in that case, the limit of the height function is trivial under the scaling we consider. We tackle this case 
separately in Subsection \ref{sec:zeralfa}

\subsection{Particles performing ASEP on an infinite line}\label{sec:infinite}

In order to complete the proof of Propositions \ref{th:fronteerlimit} and \ref{th:fronteerlimit2}, we combine Corollary \ref{th:hydrosimples} with a result that controls
the speed of particles in sparse regions: the purpose of the present section is to expose the latter result.

We consider $n$ particles performing an ASEP on the infinite line $\bbZ$ with asymmetry $(p,q)$, $n\in \lint 1,N \rint$ and we want to obtain a lower bound on the displacement of the leftmost particle.
We perform this operation in two steps: first we prove a concentration result and then we estimate the mean via stochastic
comparison with a system in the stationary state.

In this section, instead of using zeros and ones to denote presence and absence of particles, we work with the state-space
\begin{equation*}
 \gO_n:= \{ \heta=(\heta_1,\dots,\heta_n)\in \bbZ^n \ : \ \heta_1<\heta_2<\dots<\heta_n \}.
\end{equation*}
With this notation, $\heta_i$ denotes the position of the $i$-th particle starting from the left.
For the sake of using stochastic comparisons we introduce  the order 
\begin{equation}\label{stikos}
 \heta \le \heta' \quad \Leftrightarrow \quad \forall i\in \llb 1, n\rrb, \ \heta_i\le \heta'_i\;.
\end{equation}
Note that this order is the opposite of the order introduced in \eqref{eq:orderheight} at the level of height functions,
but we believe that this will not raise any confusion.

We let $\heta(t):=(\heta_1(t),\dots,\heta_n(t))$ denote the ASEP on $\bbZ$ with jump rates $p$ to the right and $q$ to the left, and
with the initial condition $\heta_i(0)=i$. Its distribution is denoted by $\bbP$.
Our aim is to show that on ``large'' time-scales (i.e. larger than $n$),
the speed of all particles is equal to what it would be in the absence of the exclusion rule: namely, $p-q$.
The following result is valid for any given sequence $(n_N)_{N\ge 1}$ satisfying $n_N\in \lint 1,N \rint$.

\begin{proposition}\label{th:partiasymp}
 Given $K>0$ there exists a constant $C(p,K)$ such that with high probability
 \begin{equation*}
 \forall t\in [0,KN]\;, \forall i \in \lint 1, n \rint\;,  \quad  |\heta_i(t)-(p-q)t|\le C \sqrt{N} \max(\sqrt{n},(\log N)^{10}).
 \end{equation*}
\end{proposition}

\subsection{Proof of Proposition \ref{th:fronteerlimit} case $\alpha=0$} \label{sec:zeralfa}

We start with the case $\alpha=0$ because it is substantially easier.

In order to obtain an upper bound for $L_{N,k}$, it is sufficient to say that the position of the leftmost particle $\eta_1$ of our original system is stochastically dominated 
by a $(p,q)$-biased simple random walk on the segment (the other $k-1$ particles to the right only slow it down). 
For the latter process, it is simple to check that, properly rescaled, it converges to $\max(t,1)$ as $N\rightarrow\infty$.

To obtain a lower bound, let us consider the ASEP $\eta$ on the segment and the ASEP $\hat\eta$ on the full line with $n=k$. Take $\heta_i(0) = i$ for every $i\in\lint 1, k \rint$. We can couple the two processes in a way that $\eta(s)\ge \hat \eta(s)$ until the time
$$\tau:=\inf\{ s\ge 0 \ : \ \hat \eta_k(s)=N\}.$$ 
Then for any fixed $t<1$ and for any $\gep>0$, Proposition \ref{th:partiasymp} (with $n=k$) implies that with high probability 
$\hat \eta_k(s)<N$ for all $s<(p-q)^{-1}Nt$, and that
$$\hat \eta_1\left(\frac{Nt}{p-q}\right)\ge N(t-\gep).$$
This implies that  $$\eta_1\left(\frac{Nt}{p-q}\right)\ge N(t-\gep).$$

For $t \ge  1$, the argument is essentially the same: one simply has to shift to the left the initial condition for $\heta$. Namely, we set $\hat \eta_i(0)=N(1-t-\gep)+i$ for $\gep>0$ small. Then, for all $s\in [0,(p-q)^{-1}Nt]$, we have $\hat \eta_k(s)<N$ with high probability. Furthermore, $\hat \eta_1\left(\frac{Nt}{p-q}\right)\ge N(1-\gep)$ with high probability so that
$$\eta_1\left(\frac{Nt}{p-q}\right)\ge N(1-\gep)\;,$$
and the lower bound follows.

\subsection{Proof of Proposition \ref{th:fronteerlimit} case $\alpha>0$}\label{sec:posalfa}

Let us fix some time horizon $t$.

In the case $\alpha>0$, we only need a lower bound on $L_{N,k}$ as the upper bound is an easy consequence of Corollary \ref{th:hydrosimples}.
We fix $\delta>0$ small and set $n= \delta N<k$.
We want to compare the first $n$ particles of $\eta(t)$ with an ASEP with $n$ particles on the infinite line $\hat \eta=(\hat \eta_1,\dots,\hat \eta_n)$.
If the initial conditions are ordered we can couple $\hat \eta$ with $\eta$ in such a way that 
\begin{equation}\label{eq:orderrr}
\forall i \in \lint 1, n \rint,\ \quad \hat \eta_i(s)\le \eta_i(s), 
\end{equation}
until the first time that $\hat \eta_n(s)=\eta_{n+1}(s)$.

Corollary \ref{th:hydrosimples} implies that with a probability tending to one we have 
 \begin{equation}\label{eq:hydrocsq}
 \forall s \in[0,t], \quad \eta_{n+1}\left(\frac{Ns}{p-q}\right)\ge N\ell_{\alpha}(s)\;.
 \end{equation}
To ensure that our coupling works until the final time $t$, we choose the initial condition for $\hat \eta$ to be much smaller than that of $\eta$. We set
\begin{equation*}
  \hat \eta_i(0)=N(\ell_{\alpha}(t)-t-\gep)+i.
\end{equation*}
 
From Proposition \ref{th:partiasymp} we have w.h.p.\ for all $i\in \lint 1,n \rint$
 \begin{equation}\label{eq:lln}
 \forall s\in\left[0,\frac{Nt}{p-q}\right]\;,\quad |\hat \eta_i(s)-(p-q)s-N(\ell_{\alpha}(t)-t-\gep)|\le C \sqrt{\delta} N.
 \end{equation}
Together with \eqref{eq:hydrocsq}, and provided $\delta$ is sufficiently small given $\gep$, this implies that the probability of the event
\begin{equation*}
 \forall s\in \left[0, \frac{Nt}{p-q}\right], \quad \hat \eta_n(s)< \eta_{n+1}(s)\;,
\end{equation*}
goes to $1$ as $N\rightarrow\infty$. Thanks to \eqref{eq:orderrr}, this implies in turn that w.h.p.\
\begin{equation*}
 \eta_1\left( \frac{Nt}{p-q} \right)\ge \hat \eta_1\left( \frac{Nt}{p-q} \right),
\end{equation*}
and thus we deduce from \eqref{eq:lln} that
\begin{equation*}
 \hat \eta_1\left( \frac{Nt}{p-q} \right)\ge N \left(\ell_{\alpha}(t)-\gep-C \sqrt{\delta}\right). 
\end{equation*}
As both $\delta$ and $\gep$ can be chosen arbitrarily small, this allows to conclude.

\subsection{Proof of Proposition \ref{th:fronteerlimit2}}

In that case, the system starts from $\xi_N$ instead of $\wedge$, so we need to adapt the arguments. 
We treat only the case where the limiting density $\alpha$ is strictly positive, the case $\alpha=0$ being simpler is left to the reader. Fix $t\ge 0$. The proof of the upper bound is simple. Either $\ell_\rho(t) > \ell+t$, and then it suffices to compare $L_N$ with a biased $(p,q)$ simple random walk as we did in Subsection \ref{sec:zeralfa}. Or $\ell_\rho(t) \leq \ell+t$, and then the upper bound is a consequence of the hydrodynamic limit stated in 
Theorem \ref{Th:HydroGene}. We turn to the proof of the lower bound. The arguments are essentially the same as those presented in Subsection \ref{sec:posalfa}, let us spell out the required modifications. The bound in \eqref{eq:hydrocsq} still holds if one replaces $\ell_\alpha(s)$ by $\ell_\rho(s)$. The initial condition has to be taken as follows:
$$ \heta_i(0) = N\big([(\ell+t)\wedge \ell_\rho(t)]-t-\gep\big) + i\;. $$
Since the speed of $\ell_\rho$ is necessarily bounded above by $1$, we deduce that w.h.p., $\heta_i(0) \le \eta_i(0)$ for all $i\in \lint 1,n\rint$. 
Then, Proposition \ref{th:partiasymp} ensures that
 \begin{equation*}
 \forall s\in\left[0,\frac{Nt}{p-q}\right]\;,\quad |\hat \eta_i(s)-(p-q)s-N\big((\ell+t)\wedge \ell_\rho(t)-t-\gep\big)|\le C \sqrt{\delta} N\;.
 \end{equation*}
The rest of the arguments then apply and we deduce that $\eta_1(t) \ge N\big((\ell+t)\wedge \ell_\rho(t)-\gep - C\sqrt{\delta}\big)$ w.h.p., 
thus concluding the proof of Proposition \ref{th:fronteerlimit2} in the case $\alpha > 0$.\\

\subsection{Proof of Proposition \ref{th:partiasymp}}

Since the system is ordered, we only need to prove that the following two inequalities hold w.h.p.
 \begin{equation}\label{eq:devs}\begin{split}
 \forall t\in [0,KN], \quad \heta_1(t)\ge (p-q)t - C \sqrt{N}\max(\sqrt{n}, (\log N)^{10}),\\
 \forall t\in [0,KN], \quad \heta_n(t)\le (p-q)t + C \sqrt{N}\max(\sqrt{n}, (\log N)^{10}).
\end{split} \end{equation}

The proof of the second inequality is in fact very similar to the proof of the first one. Hence we decide to 
discuss in detail only the case of $\heta_1$ and we explain briefly the needed modifications for $\heta_n$ when they are non trivial.

The proof of the result is decomposed into two separate statements: first we show that $\heta_1(nt)$ is concentrated around its mean using a martingale concentration
result from \cite{LesVol01}, 
and then we obtain a lower bound on $\bbE[\heta_1(nt)]$ by comparing the system with a stationary one.

\begin{lemma}\label{conct}
Under the assumptions above, there exists $c>0$ such that for all $n\geq 2$, all $t\ge 0$, all $i\in \lint 1,n\rint$ 
and all $u\geq 0$ such that $u^2/(t (\log n)^2) > 1$, we have
\begin{equation*}
  \bbP\left[ |\heta_i(t)-\bbE[\heta_i(t)]|\ge  u \right]\le 2\exp\left(-c\left(\frac{u^2}{t(\log n)^2}\right)^{1/3}\right). 
\end{equation*}
\end{lemma}

\begin{lemma}\label{damean}
 With the assumptions above we have 
 \begin{equation*}\begin{split}
 \bbE[\heta_1(t)]&\ge  (p-q)t- 2\max(n,\sqrt{tn})\;,\\
 \bbE[\heta_n(t)]&\le  (p-q)t+ 2\max(n,\sqrt{tn})\;.
\end{split} \end{equation*}
\end{lemma}

We postpone the proofs of the lemmas to the end of the subsection and we proceed to the proof of \eqref{eq:devs}. Combining the two lemmas, we obtain easily that for any $t\in [0,KN]$ we have 

\begin{equation*}
 \bbP\left[ \heta_1(t)\le  (p-q)t - C \sqrt{N}\max\left(\sqrt{n},(\log N)^4\right) \right] \le \exp(-c (\log N)^2)\;,
\end{equation*}
for some constants $c,C >0$. Hence, the probability that there exists $t\in \big\{jN^{-4}, j\in \lint 1, KN^5\rint\big\}$ such that
\begin{equation*}
\heta_1(t)\le  (p-q)t - C \sqrt{N}\max\left(\sqrt{n},(\log N)^4\right) \;,
\end{equation*}
is bounded by $KN^5 \exp(-c (\log N)^2)$ which vanishes as $N\rightarrow\infty$. Then, we notice that the probability that there exists an interval $[jN^{-4}, (j+1)N^{-4})$ on which $\heta_1$ makes more than one jump is bounded by a term of order $N^{5}$ times the probability that a Poisson clock rings more than once in a time interval of length $N^{-4}$, that is, by a term of order $N^{-3}$. We deduce that the probability that there exists $t\in [0,KN^5]$ such that
\begin{equation*}
\heta_1(t)\le  (p-q)t - C \sqrt{N}\max\left(\sqrt{n},(\log N)^4\right)-1 \;,
\end{equation*}
is vanishing with $N$, from which we deduce \eqref{eq:devs} for $\heta_1$. A very similar argument yields \eqref{eq:devs} for $\heta_n$.

\begin{proof}[Proof of Lemma \ref{conct}]
First, let us consider the case where $t$ is an integer. Fix $i\in \lint 1,n\rint$. For such a $t$, we define the martingale  $(M^t_s, s\in \lint 0,t\rint)$ by 
\begin{equation*}
M^t_s:= \bbE[\heta_i(t) \ | \ \cF_s] - \bbE[\heta_i(t)]\;,\quad \cF_s:=\sigma\big(\heta(u), u\in[0,s] \big)\;.
\end{equation*}

We are going to prove tail bounds on the increments of $M^t$ to obtain concentration. For convenience, we set $\Delta M^t_s := M^t_s - M^t_{s-1}$ for any
$s\in\lint 1,t\rint$. We let $\bar R(s)$, resp.\ $\bar L(s)$, denote the maximal number of jumps to the right, resp.\ left, performed by a particle in the system during the time interval $(s-1,s]$,
\begin{equation*}\begin{split}
        \bar R(s)&:=\max_{i\in \llb 1,n \rrb} \#\{\ t\in(s-1,s] \ : \ \heta_i(t)=\heta_i(t_{-})+1 \ \}=\max_{i\in \llb 1,n \rrb} R_i(s)\;,\\        
        \bar L(s)&:=\max_{i\in \llb 1,n \rrb} \#\{\ t\in(s-1,s] \ : \ \heta_i(t)=\heta_i(t_{-})-1\ \}=\max_{i\in \llb 1,n \rrb} L_i(s)\;.  
                \end{split}
\end{equation*}
At time $s$ we have 
$$ \heta(s-1)-  \bar L(s)  \le \heta(s)\le \heta(s-1)+  \bar R(s).$$
where for $k\in \bbN$ and $\heta \in \gO_n$,  
$$\heta+k:=(\heta_1+k,\dots,\heta_n+k)\;.$$
Let us now consider two initial conditions given by $\heta(s-1)$ on the one hand, and $\heta(s)+\bar L(s)$ on the other hand, and let us run the ASEP dynamics for both systems for a time length $t-s$: the two configurations are stochastically ordered, and their laws coincide with the laws of $\heta(t-1)$ conditionally given $\cF_{s-1}$ for the first one and of $\heta(t)+\bar L(s)$ conditionally given $\cF_s$ for the second one. A similar reasoning can be applied to $\heta(s-1)$ and $\heta(s)+\bar R(s)$. Therefore, we get for all $t\geq s$
\begin{equation*}\begin{split}
\bbE[\heta_i(t) \ | \ \cF_s]&\le \bbE[\heta_i(t-1) \ | \ \cF_{s-1}]+\bar R(s)\;,\\
\bbE[\heta_i(t) \ | \ \cF_s]&\ge \bbE[\heta_i(t-1) \ | \ \cF_{s-1}]-\bar L(s)\;.
\end{split}\end{equation*}
Furthermore as we have  
$-L_i(t) \le \heta_i(t)-\heta_i(t-1) \le R_i(t)$, and as 
both variables    $R_i(t)$ and  $L_i(t)$ are, conditionally to $\cF_{t-1}\supset \cF_{s-1}$, dominated by Poisson random variables of parameters $p$ and $q$, we have for any $s\le t$,
\begin{equation*}
  |\bbE[\heta_i(t) \ | \ \cF_{s-1}]-\bbE[\heta_i(t-1) \ | \ \cF_{s-1}]|\le \max(p,q)=p\;.
\end{equation*}
Then, we write
\begin{align*}
\big| \Delta M^t_s\big| &\le \big|\bbE[\heta_i(t) \ | \ \cF_{s-1}]-\bbE[\heta_i(t-1) \ | \ \cF_{s-1}] \big|\\
&\quad+ \big|\bbE[\heta_i(t) \ | \ \cF_{s}]-\bbE[\heta_i(t-1) \ | \ \cF_{s-1}] \big|\\
&\le p + \max( \bar R(s), \bar L(s))\;.
\end{align*}
As $\bar R(s)$ and $\bar L(s)$ are bounded above by the maxima of $n$ Poisson variables of mean $p$ and $q$, we obtain that
for any constant $c_p>0$ there exists $C>0$ such that
\begin{equation}\label{Eq:BoundDeltaM}
 \bbP[|\Delta M_s^t| \ge u]\le Cn e^{-c_p u}\;,\quad \forall u\geq 0\;.
\end{equation}
This implies the existence of $K>0$, independent of $n$, such that
$$\bbE\left[e^{(\log n)^{-1}\Delta M_s^t}\right]\le K\;.$$
Hence, we can apply \cite[Theorem 3.2]{LesVol01} (which is simply Azuma's inequality combined with some truncation argument for the increments)
to the martingale $(\log n)^{-1}M_s^t$. More precisely, we apply the bound obtained at the end of the proof of Theorem 3.2 therein (equation right below (11)) and deduce that the asserted concentration estimate holds.\\
So far, we have proven the bound when $t\in \bbN$. To treat the general case $t\geq 0$, it suffices to bound the increment $M^t_t - M^t_{\lfloor t \rfloor}$. Inspecting the arguments above, we observe that \eqref{Eq:BoundDeltaM} still holds in that case, so that the proof carries through.
\end{proof}

\begin{proof}[Proof of Lemma \ref{damean}]
We notice that adding particles to the right, reducing the drift of some particles, or changing the initial condition by shifting the particles to the left have the effect of slowing down $\heta_1(t)$ in the sense that the system obtained after such modifications is dominated by the original one.

We consider more specifically the following modification of the dynamics which we call $\tilde \eta$:
\begin{itemize}
 \item  $\tilde\eta_n(0)=n$ and $(\tilde \eta_{i+1}-\tilde \eta_i)_{i=1,\ldots,n-1}$ 
 are IID geometric random variables of parameter $\mu<1$, that is, $P[\tilde \eta_{i+1}-\tilde \eta_i=k]=(1-\mu)\mu^{k-1}$,
 \item  The jump rate to the left is still $q$ but the jump rate to the right is $p$ for the first (leftmost) particle, $p_1$ for particles labeled from $2$ to $n-1$ and $p_2$ for the last (rightmost) one.
\end{itemize}
It can be checked that the product of geometric laws with parameter $\mu$ is stationary (but not reversible in general) for the Markov chain 
$\left[(\tilde \eta_{i+1}(t)-\tilde \eta_i(t))_{i=1}^{n-1}\right]_{t\ge 0}$, provided that
\begin{equation*}
 \mu(p+q)=\mu p_1+q=p_2+q\;.
\end{equation*}
Note that this implies $p_2\le p_1 \le p$, and therefore as we have $\heta(0)\ge \tilde \eta(0)$ this implies that there exists a coupling such that
\begin{equation}\label{eq:campari}
\forall t \ge 0, \quad \heta(t) \ge  \tilde \eta(t)\;.
\end{equation}
As the increments are stationary, the expected drift of the first particle is the same as the initial drift and thus
$$\forall t\ge 0, \partial_t \bbE[\tilde \eta_1(t)]=p\mu- q\;.$$ 
Moreover 
\begin{equation*}
\bbE[\tilde \eta_1(0)]=1-(n-1)\frac{\mu}{1-\mu}.
\end{equation*}
Using \eqref{eq:campari} this implies that
\begin{equation}\label{lvb}
\bbE[\heta_1(t)]\ge 1+(p\mu- q) t- \frac{\mu (n-1)}{1-\mu}\ge (p-q) t -(1-\mu) t- \frac{n}{1-\mu}.
\end{equation}
Choosing $\mu$ such that $1-\mu=\min(1, \sqrt{n/t})$ we obtain 
\begin{equation*}
\bbE[\heta_1(t)]\ge (p-q)t- 2\max(n, \sqrt{nt})\;,
\end{equation*}
as required.

To establish the asserted result for the rightmost particle, we consider an analogous system where particle spacings have the same initial distribution (IID geometric random variables 
with parameter $\mu$),
but we fix $\tilde \eta_1(0)=1$ and set the jump rate to the right to be $p_4$ for the first particle, $p_3$ for the particles with labels from $2$ to $n-1$ and $p$ for the particle with label $n$, where 
\begin{equation*}
 \mu(p_4+q)=\mu p_3+q=p+q\;.
\end{equation*}
We can couple $\heta$ and $\tilde \eta$ in a way such that $\tilde \eta(t)\ge  \heta(t)$ for all $t\geq 0$. The speed of the rightmost particle is given by $p-\mu q$, and its 
initial mean is $1+(n-1)\frac{1}{1-\mu}$. Taking $1-\mu = \min(1,\sqrt{n/t})$ yields the bound
$$ \bbE[\heta_n(t)] \le (p-q)t + 2 \max(n,\sqrt{nt})\;,$$
as required.
\end{proof}

{\bf Acknowledgements:} The present work was partially realized during a visit of H.L.\ to CEREMADE, he acknowledges kind hospitality and support.
H.L.\ also acknowledges the support of a productivity grant from CNPq.

\appendix

\section{Grand coupling}\label{Appendix:Coupling}

We are given two collections $(P_i, i\in\lint 1, N-1 \rint)$ and $(Q_i,i\in\lint 1, N-1 \rint)$ of independent Poisson processes
with jump rates $p$ and $q$ respectively. The grand coupling for the biased card shuffling is defined as follows. For any initial
condition $\xi \in \cS_N$, the process $\sigma^\xi$ starts from $\sigma^\xi_0=\xi$ and is piecewise constant outside of the jump times of $P_i$ and $Q_i$.
The transition at these latter time are defined as follows:
At every jump time $s > 0$ of $P_i$ we place
the cards at sites $i,i+1$ in the increasing order, that is, $\sigma^\xi_s = \sigma^\xi_{s-} \circ \tau_i$  if $\sigma^\xi_{s-}(i) > \sigma^\xi_{s-}(i+1)$
and $\sigma^\xi_s = \sigma^\xi_{s-}$ otherwise. At every jump time $s >0$ of $Q_i$ we place the cards at sites $i,i+1$ in the decreasing order, that is, $\sigma^\xi_s = \sigma^\xi_{s-} \circ \tau_i$  
if $\sigma^\xi_{s-}(i) < \sigma^\xi_{s-}(i+1)$
and $\sigma^\xi_s = \sigma^\xi_{s-}$ otherwise.

\medskip

Taking the image of this process through the maps $h_k:\cS_N\rightarrow \Omega_{N,k}$ for $k\in \llb 1,N-1\rrb$, we get a grand coupling of the asymmetric simple exclusion processes. The dynamics at the level of the height functions can be restated as follows: at a jump time of $P_i$, if the height function makes an upwards corner at site $i$ then it flips into a downwards corner; similarly, at a jump time of $Q_i$, if the height function makes a downwards corner at site $i$ then it flips into an upwards corner.

Let us check that the dynamics preserves the order. To that end, it suffices to check that all the transitions do so. Consider a jump time of $P_i$ and suppose we are given two heights functions $h\leq h'$ right before the jump time. If both $h$ and $h'$ (or none of them) have an upwards corner at site $i$, then both flip downwards and the ordering is preserved. If only $h$ has an upwards corner, then the flip can only make $h$ smaller and therefore the ordering is also preserved upon the jump. Let us now suppose that only $h'$ has an upwards corner at site $i$. Inspecting the possible shapes for $h$ at site $i$, and recalling that the set of possible values for the height function at a given site $i$ has a span equal to $2$ we deduce that necessarily $h(i) \leq h'(i)+2$. Therefore, after the jump the ordering is still preserved at site $i$. By symmetry, the arguments are the same for upwards flips. This concludes the proof.

\bibliographystyle{imsart-nameyear}
\bibliography{library}

\end{document}